\documentclass[11pt,reqno]{amsart}
\usepackage{amssymb,mathrsfs,graphicx,xcolor}
\usepackage{mathtools}
\usepackage{enumerate}
\usepackage[margin=3cm]{geometry}
\mathtoolsset{showonlyrefs=true}

\def\@abssec#1{\vspace{.05in}\footnotesize \parindent .2in
{\bf #1. }\ignorespaces}

\newtheorem{theorem}{Theorem}[section]

\newtheorem{lemma}[theorem]{Lemma}
\newtheorem{proposition}[theorem]{Proposition}

\newtheorem{remark}[theorem]{Remark}

\newcounter{hypo}

\newcommand{\R}{\ensuremath{\mathbb{R}}}

\newcommand{\T}{\ensuremath{\mathbb{T}}}

\newcommand{\beqs}{\begin{equation*}}
\newcommand{\eeqs}{\end{equation*}}

\renewcommand{\div}{{\rm div}\,}

\newcommand{\tend}{\bibliographystyle{plain}\bibliography{ccrituniq}\end{document}}
\newcommand{\dd}{\mathrm{d}}

\allowdisplaybreaks \numberwithin{equation}{section}

\begin{document}

\title[Uni-directional Euler-alignment system]{Global well-posedness and refined regularity criterion for the uni-directional Euler-alignment system}

\author{Yatao Li}
\address{Department of Mathematics, Jiangxi University of Finance and Economics, Nanchang 330032, P. R. China
AND Laboratory of Mathematics and Complex Systems (MOE), School of Mathematical Sciences, Beijing Normal University, Beijing 100875, P.R. China}
\email{liyatao$\_$maths@163.com}
\author{Qianyun Miao}
\address{School of Mathematics and Statistics, Beijing Institute of Technology, Beijing 100081, P. R. China}
\email{qianyunm@bit.edu.cn}
\author{Changhui Tan}
\address{Department of Mathematics, University of South Carolina, Columbia SC 29208, USA}
\email{tan@math.sc.edu}
\author{Liutang Xue}
\address{Laboratory of Mathematics and Complex Systems (MOE), School of Mathematical Sciences, Beijing Normal University, Beijing 100875, P.R. China}
\email{xuelt@bnu.edu.cn}

\date{\today}
\maketitle

\begin{abstract}
We investigate global solutions to the Euler-alignment system in $d$ dimensions with unidirectional flows and strongly singular communication protocols $\phi(x) = |x|^{-(d+\alpha)}$ for $\alpha \in (0,2)$. Our paper establishes global regularity results in both the subcritical regime $1<\alpha<2$ and the critical regime $\alpha=1$. Notably, when $\alpha=1$, the system exhibits a critical scaling similar to the critical quasi-geostrophic equation. To achieve global well-posedness, we employ a novel method based on propagating the modulus of continuity. Our approach introduces the concept of simultaneously propagating multiple moduli of continuity, which allows us to effectively handle the system of two equations with critical scaling. Additionally, we improve the regularity criteria for solutions to this system in the supercritical regime $0<\alpha<1$.
\end{abstract}

\tableofcontents

\section{Introduction}
In this paper, we consider the hydrodynamic Euler-alignment system described by the following equations
\begin{equation}\label{eq.EAS}
\begin{split}
\begin{cases}
  ~~\partial_t \rho + \div( \rho \mathbf{u})= 0, \\[3pt]
  ~~\partial_t\mathbf{ u} + \mathbf{u}\cdot\nabla \mathbf{u }=\displaystyle\int_{\R^d}\phi(x-y)(u(y)-u(x))\rho(y)\mathrm dy,
\end{cases}
\end{split}
\end{equation}
for $(x,t)\in \R^d\times \R_{+}$, subject to the initial condition
\begin{align*}
  (\rho, \mathbf{u})|_{t=0}(x) = (\rho_0,\mathbf{u}_0)(x).
\end{align*}
Here, $\rho$ and $\mathbf{u}=(u^1(x,t),\cdots,u^d(x,t))$ represent the density and velocity vector field, respectively. The second equation of \eqref{eq.EAS} includes the \emph{alignment force}, which is determined by the \emph{communication protocol} $\phi$, that measures the strength of the alignment interactions and is assumed to be non-negative and radially decreasing. The alignment force can be expressed as a commutator:
\begin{equation}\label{eq:alignment}
\int_{\R^d}\phi(x-y)(u(y)-u(x))\rho(y)\mathrm dy=-[\mathcal{L}_{\phi}, \mathbf{u}]\,\rho:=-\mathcal{L}_{\phi}(\rho \mathbf{u})+\mathcal{L}_{\phi}(\rho) \mathbf{u},		
\end{equation}
where
\[
  \mathcal{L}_{\phi}(f)(x) := \int_{\R^d} \phi(x-y)(f(x)-f(y))\mathrm dy.
\]

The system \eqref{eq.EAS} can be  seen as a macroscopic representation of the well-known \emph{Cucker-Smale flocking model} \cite{cucker2007emergent}
\begin{align*}
\begin{split}
\begin{cases}
    ~~\mathbf{\dot x}_i=\mathbf{v}_i,\\[3pt]
    ~~\mathbf{\dot v}_i=\displaystyle\frac 1N\sum_{j=1}^N\phi(\mathbf{x}_i-\mathbf{x}_j)(\mathbf{v}_j-\mathbf{v}_i),
\end{cases}
\end{split}
\text{where }  \big(\mathbf{x}_i(t), \mathbf{v}_i(t)\big)\in \R^d\times\R^d,\quad
i=1,\cdots,N,
\end{align*}
which describes the collective motion of $N$ agents adjusting their velocities based on a weighted average of their neighbors. For a detailed derivation of \eqref{eq.EAS} and related discussions, we refer the readers to \cite{carrillo2017review, ha2008particle,shvydkoy2021dynamics} and the references therein.

Our main focus is on the global well-posedness and asymptotic behaviors of the Euler-alignment system \eqref{eq.EAS}. Extensive progress has been made in recent years, revealing that different types of communication protocols lead to different system behaviors. For bounded and Lipschitz communication protocols, the alignment force \eqref{eq:alignment} acts as a nonlocal damping mechanism. This results in a \emph{critical threshold phenomenon}: subcritical initial data lead to global well-posedness, while supercritical initial data lead to the formation of finite-time singularities. See e.g. \cite{carrillo2016critical,tadmor2014critical}. A similar theory has been established for weakly singular communication protocols, where $\phi$ is unbounded but integrable at the origin. See e.g. \cite{tan2020euler}.

Another type of communication protocol that is of particular interest to us is when $\phi$ is \emph{strongly singular}, meaning it is not integrable at the origin. A prototype example of such a protocol is given by:
\[\phi(z)= c_\alpha |z|^{-(d+\alpha)},\quad c_\alpha=\tfrac{2^\alpha\Gamma(\frac{d+\alpha}2)}{\pi^{\frac{d}{2}}\Gamma(-\frac{\alpha}{2})},\quad \alpha\in(0,2),\]
when the operator $\mathcal{L}_{\phi}$ is characterized by the fractional Laplacian:
\begin{align*}
  \mathcal{L}_{\phi}(f)(x)=\Lambda^\alpha f(x):=c_\alpha\mathrm {p.v.}\int_{\R^d}\frac {f(x)-f(y)}{|x-y|^{d+\alpha}}\dd y.
\end{align*}

The singularity in the communication protocol induces dissipation (or ellipticity) in the alignment force \eqref{eq:alignment}, resulting in a regularization effect on the solutions of \eqref{eq.EAS}. This phenomenon has been the subject of extensive research, especially in the context of a one-dimensional periodic domain $\T$. Notably, studies conducted in \cite{do2018global,shvydkoy2017eulerian,shvydkoy2017eulerian2,shvydkoy2018eulerian} have demonstrated that for any smooth non-vacuous initial data, global smooth solutions arise. Furthermore, these results have been extended to encompass general strongly singular communication protocols \cite{kiselev2018global}, as well as scenarios involving misalignment in communications \cite{miao2021global}. When the initial data contain a vacuum, the ellipticity becomes degenerate, leading to the possibility of finite-time singularity formations \cite{arnaiz2021singularity,tan2019singularity}.

The remarkable success of the theory in one dimension can be largely attributed to the presence of a conserved auxiliary quantity:
\[ G=\partial_x u-\Lambda^\alpha\rho,\]
which satisfies the continuity equation
\begin{equation*}
  \partial_tG+\partial_x(Gu)=0.
\end{equation*}
The conservation of $G$ plays a crucial role in the analysis and understanding of the system dynamics. In particular, when the initial condition $G_0$ is identically zero ($G_0\equiv0$), it follows that $G\equiv0$, and \eqref{eq.EAS} reduces to the following advection-diffusion equation:
\begin{equation}\label{eq:1DG0}
\partial_t\rho+u\partial_x\rho=-\rho\Lambda^\alpha\rho,\quad u=\partial_x^{-1}\Lambda^\alpha\rho=-\partial_x\Lambda^{\alpha-2}\rho.	
\end{equation}
This equation is recognized and extensively studied as a model for one-dimensional nonlinear porous medium flow with fractional potential pressure \cite{caffarelli2013regularity,caffarelli2011nonlinear,caffarelli2016regularity}. Furthermore, in the special case where $\alpha=1$, \eqref{eq:1DG0} corresponds to a one-dimensional model of the two-dimensional critical quasi-geostrophic equation, which has been investigated in \cite{chae2005finite}.

However, extending the theory to higher dimensions has proven to be challenging and has not yielded comparable success. A natural replacement for the auxiliary quantity in higher dimensions is given by:
\[G:=\nabla\cdot\mathbf{u}-\Lambda^\alpha\rho\]
which satisfies the equation:
\begin{align*}
  \partial_t G + \nabla\cdot(G\mathbf{u})=(\nabla\cdot\mathbf{u})^2-\mathrm{Tr}[(\nabla\mathbf{u})^2].
\end{align*}
However, this new formulation is no longer conservative, as the right-hand side is not necessarily zero. The absence of a conserved quantity in higher dimensions poses a significant challenge in extending the results obtained in one dimension. In the multi-dimensional case, the global well-posedness result remains incomplete and typically requires additional smallness assumptions on the initial data. For instance, when the initial velocity amplitude is small relative to its higher-order norms, Shvydkoy \cite{shvydkoy2019global} established global existence and stability results for nearly aligned flocks. Additionally, Danchin et al. \cite{danchin2019regular} demonstrated global well-posedness for solutions to \eqref{eq.EAS} within the critical Besov space framework, under the assumption that the initial data $(\rho_0,\mathbf{u}_0)$ is sufficiently close to the constant state $(1,\mathbf{0})$ in terms of Besov space norms.

Recently, Lear and Shvydkoy introduced a class of uni-directional flows in their work \cite{lear2021unidirectional}. This class of flows is given by:
\begin{align}\label{uni.u}
\mathbf{u}(x,t)=u(x,t)\mathbf{d},\quad\mathbf{d}\in\mathbb{S}^{d-1},\quad u(x,t):\R^d\times \R_+\rightarrow\R.
\end{align}
It can be observed that the structure \eqref{uni.u} is preserved over time by the Euler-alignment system \eqref{eq.EAS}. Moreover, under the uni-directional flow condition \eqref{uni.u}, the term $(\nabla\cdot\mathbf{u})^2-\mathrm{Tr}[(\nabla\mathbf{u})^2]$ in the equation for $G$ vanishes, leading to the conservation of $G$.

Without loss of generality, considering the rotational invariance of system \eqref{eq.EAS}, we can assume that $\mathbf{d}=(1,0,\cdots,0)$ corresponds to the $x_1$ direction. This leads to the following system:
\begin{equation}\label{eq:Euni}
\begin{split}
\begin{cases}
  ~~\partial_t \rho + \partial_{x_1}( \rho u)= 0, \\
  ~~\partial_tu + u\,\partial_{x_1}{u }=-\Lambda^\alpha\big(\rho {u}\big) + (\Lambda^\alpha\rho) u
 = :\mathcal{C}_\alpha(u,\rho), \\
  ~~(\rho,  {u})|_{t=0}(x) = (\rho_0, {u}_0)(x),
\end{cases}
\end{split}
\end{equation}

Although this system exhibits the characteristics of one-dimensional flow, it is important to note that the spatial variable $x$ still belongs to $\R^d$, and the dissipation is in $d$ dimensions. Thus, it differs from the traditional one-dimensional Euler-alignment system. However, we recall the aforementioned feature of system \eqref{eq:Euni}, namely the conservation of the auxiliary quantity:
\begin{equation}\label{eq:G}
G=\partial_{x_1}u- \Lambda^\alpha \rho, \qquad \partial_t G+\partial_{x_1}(Gu)=0.
\end{equation}

Based on this conservation property, it is reasonable to inquire whether the system \eqref{eq:Euni} possesses a similar global well-posedness theory as the one-dimensional system. However, this is not the case.

To illustrate this, consider the special case when $G_0\equiv0$, resulting in $G\equiv0$ throughout the system. In this scenario, \eqref{eq:Euni} reduces to the advection-diffusion equation:
\begin{equation}\label{eq:G0}
\partial_t\rho+u\partial_{x_1}\rho=-\rho\Lambda^\alpha\rho,\quad u=\partial_{x_1}^{-1}\Lambda^\alpha\rho.
\end{equation}
In contrast to the one-dimensional system \eqref{eq:1DG0}, where the regularity of $u$ can be controlled by the regularity of $\rho$ through their relationship, in \eqref{eq:G0}, only $\partial_{x_1}u$ can be controlled by the regularity of $\rho$. There is no direct mechanism to control the other partial derivatives of $u$ based on its relation to $\rho$.

Instead, one may approach the system \eqref{eq:Euni} by focusing directly on the $u$ equation in $\eqref{eq:Euni}_2$ and investigate the regularization effect of the alignment force $\mathcal{C}_\alpha(u,\rho)$. By enforcing $\rho\equiv1$, we observe that
\[\mathcal{C}_\alpha(u,1)=-\Lambda^\alpha u.\]
In this case, the equation $\eqref{eq:Euni}_2$  becomes the fractal Burgers equation
\begin{equation}\label{eq:fBurgers}
 \partial_t u+ u\,\partial_{x_1}u = -\Lambda^\alpha u,
\end{equation}
 which has been extensively studied in \cite{kiselev2008blow}. The behavior of solutions depends on the value of $\alpha$. When $\alpha\in(1,2)$, the dissipation dominates, resulting in globally well-posed solutions. However, when $\alpha\in(0,1)$, the dissipation is not strong enough, and finite-time singularity formations may occur.
The critical case arises when $\alpha=1$, and it is particularly subtle to analyze. Global well-posedness has been established using a novel method based on the modulus of continuity. This approach was invented by Kiselev et al. in their celebrated work on the critical quasi-geostrophic equations \cite{kiselev2007global}, and has been successfully used to analyze many equations with critical scalings, e.g. \cite{kiselev2011nonlocal,kiselev2018global,kiselev2022global,miao2012global,miao2021global}.

The uni-directional Euler-alignment system \eqref{eq:Euni} has been thoroughly investigated by Lear and Shvydkoy in their work \cite{lear2021unidirectional} for the case of $\alpha\in(1,2)$. They establish that the alignment force $\mathcal{C}_\alpha(u,\rho)$, which behaves similarly to the fractional Laplacian $-\Lambda^\alpha u$, dominates the Burgers nonlinearity, leading to global well-posedness. Their approach builds upon the H\"older regularization results developed in \cite{silvestre2012holder,schwab2016regularity}.

However, in the critical case of $\alpha=1$, the intricate structure of $\mathcal{C}_1(u,\rho)$ presents significant challenges in extracting sufficient dissipation to counterbalance the nonlinear advection. To the best of our knowledge, the only available result in the literature is provided by Lear in \cite{lear2021global}, where global well-posedness is established for the specific case \eqref{eq:G0}. For general equation \eqref{eq:Euni}, smallness assumptions are required to obtain global smooth solutions.

Now, we present our first main result on the global well-posedness of the system \eqref{eq:Euni} for $1\leq\alpha<2$.
\begin{theorem}[Global well-posedness]\label{thm.main}
Let $1\leq \alpha<2$ and $(\rho_0,u_0)\in H^{m+\alpha}(\T^d)\times H^{m+1}(\T^d)$, where $m>\frac d2+1$ and $\rho_0(x)>0$.
Then there exists a global unique non-vacuous solution $(\rho,u)$ to the uni-directional Euler-alignment system \eqref{eq:Euni}
in the following class:
\[
\rho\in C_w(\R_+; H^{m+\alpha}(\T^d)),\quad u\in C_w(\R_+; H^{m+1}(\T^d))\cap L^2 (\R_+; \dot H^{m+1+\frac\alpha2}(\T^d)).
\]
\end{theorem}

When $\alpha\in(1,2)$, our theorem recasts the results presented in \cite{lear2021unidirectional}, but through an alternative approach based on the method of modulus of continuity.

The main contribution of this theorem lies in the critical case when $\alpha=1$. Our result establishes global regularity without imposing any smallness assumptions. Overcoming this challenge requires extracting sufficient dissipation from the alignment force $\mathcal{C}_1(u,\rho)$. A major difficulty arises from the system's invariance under the critical scaling:
\begin{equation}\label{eq:criticalscaling}
  \rho(x,t)\rightsquigarrow \rho(\lambda x,\lambda t),\quad
  u(x,t)\rightsquigarrow u(\lambda x,\lambda t), \quad \forall~\lambda >0.
\end{equation}
As a consequence, energy-based or scaling-based estimates alone are inadequate to ensure global regularity. Instead, we employ the method of modulus of continuity, which draws inspiration from the approach used in \cite{kiselev2008blow} on the critical fractal Burgers equation \eqref{eq:fBurgers}.

We would like to emphasize an additional major difficulty in the analysis. Unlike the linear fractional dissipation term $\mathcal{C}_1(u,1)=-\Lambda u$, the alignment force $\mathcal{C}_1(u,\rho)$ is highly nonlinear and dependent on the density $\rho$. In particular, the most dangerous term is the difference between $\mathcal{C}_1(u,\rho)$ and $\rho\,\mathcal{C}_1(u,1)$ namely
\begin{equation}\label{eq:difference}
	\mathcal{C}_1(u,\rho)-\rho\,\mathcal{C}_1(u,1)=c_\alpha\mathrm {p.v.}\int_{\R^d}\frac{(\rho(y)-\rho(x))(u(y)-u(x))}{|x-y|^{d+1}}\,\dd x,
\end{equation}
which cannot be solely controlled by the linear dissipation $-\Lambda u$. Additional a priori control on the regularity of $\rho$ is required. However, utilizing the relation
\begin{equation}\label{eq:rhorelation}
 \rho = \partial_{x_1}\Lambda^{-1}u-\Lambda^{-1}G
\end{equation}
does not provide a sufficient bound. The main obstacle is the lack of $L^\infty$ to $L^\infty$ bound for the Reisz transform $\partial_{x_1}\Lambda^{-1}$. Consequently, standard approaches employed in \cite{do2018global,kiselev2008blow,lear2021unidirectional} do not yield the desired global well-posedness result.

To overcome this difficulty, we propose a new concept of \emph{simultaneously propagating two moduli of continuity}. In addition to propagating the modulus of continuity on $u$, as done in \cite{kiselev2008blow}, we simultaneously propagate a modulus of continuity on $\rho$ through the equation \eqref{eq:Euni}. The key lies in smartly choosing a modulus of continuity for $\rho$ that is stronger than what can be obtained solely through the relation \eqref{eq:rhorelation}. This choice allows us to achieve sufficient control over the term \eqref{eq:difference}.

We believe that this new approach represents an extension of the method of modulus of continuity and opens up possibilities for studying \emph{systems} of equations with critical scalings. By simultaneously propagating multiple moduli of continuity, we can effectively handle the nonlinear interactions and dependencies in the system, leading to the desired global well-posedness results. This innovative approach may pave the way for further developments in the analysis of critical systems and their regularity properties.

Our next result concerns the asymptotic \emph{flocking} behavior of solutions to \eqref{eq:Euni}. This phenomenon has been extensively studied in the general context of the Euler-alignment system \eqref{eq.EAS} (see, e.g., \cite{lear2021global,lear2022geometric,lear2021unidirectional, lear2022existence, leslie2019structure, leslie2023sticky,shvydkoy2017eulerian2, tadmor2014critical}).
In particular, the global solution tends to exhibit certain collective behavior. Specifically, the velocity $u$ converges to its average value $\bar{u}$, given by
\begin{equation}\label{def:bar-u}
\bar{u}:=\frac{\int_\T(\rho_0u_0)(x)\mathrm dx}{\int_\T\rho_0(x)\mathrm dx},
\end{equation}
while the density profile tends to a traveling wave flocking state:
\begin{equation*}
\rho(x,t)\rightarrow\rho_{\infty}(x-\bar{u}t).
\end{equation*}

We establish the following result:
\begin{theorem}[Asymptotic behavior]\label{thm.flocking}
Let $(\rho,u)$ be the global solution to \eqref{eq:Euni} as guaranteed by Theorem \ref{thm.main}. Then we have
\begin{align}\label{eq:flock1}
  \| u(t)-\bar{u}\|_{W^{1,\infty}}
  \leq Ce^{-c_0\, t},\quad \forall~ t>0,
\end{align}
where $\bar{u}$ is defined as in \eqref{def:bar-u}, and the rate $c_0>0$ depends only on $\alpha,d$ and $\bar{u}$.
Consequently, there exist $\rho_\infty\in H^{m+\alpha}(\T^d)$ such that
\begin{align}\label{eq:flock2}
  \| \rho(\cdot,t)-\rho_\infty(\cdot-\bar{u}t)\|_{C^\beta}
  \leq Ce^{-c_0\, t},\quad \forall~ t>0,\,\,0< \beta <1,
\end{align}
\end{theorem}

The exponential decays observed in \eqref{eq:flock1} and \eqref{eq:flock2} are commonly referred to as \emph{strong flocking}. This result has already been established and documented in the literature, for instance in \cite{lear2021global,lear2021unidirectional}.

In our analysis, we introduce a time-dependent modulus of continuity on $u$, inspired by the approach presented in \cite{kiselev2011nonlocal}. This innovative technique enables us to derive the asymptotic behavior \eqref{eq:flock1} as a complementary result to the global well-posedness theorem. 

Our final result focuses on the system \eqref{eq:Euni} with $\alpha\in(0,1)$. In the context of the fractional Burgers equation \eqref{eq:fBurgers}, it is well-known that the dissipation term $\mathcal{C}_\alpha(u,1)$ is not sufficiently strong to prevent the formation of singularities within finite time. However, a remarkable discovery in \cite{do2018global} demonstrated that the alignment force $\mathcal{C}_\alpha(u,\rho)$, which incorporates the density $\rho$ as a weight, actually enhances the dissipation for the one-dimensional Euler-alignment system, yielding global regularity. The natural question that arises is whether a similar phenomenon can be observed in multi-dimensional systems. Specifically, for the uni-directional flow described by \eqref{eq:Euni}, it remains unclear whether the enhanced dissipation effect occurs solely in the $x_1$ direction, as suggested by \eqref{eq:G0}.

Although it is uncertain whether the dissipation induced by $\mathcal{C}_\alpha(u,\rho)$ can surpass that of $\mathcal{C}_\alpha(u,1)$, our subsequent result demonstrates that they are at least comparable. The following theorem provides a refined regularity criterion for the system \eqref{eq:Euni} when $\alpha\in(0,1)$.

\begin{theorem}[Regularity criterion]\label{thm:reg-cr}
Let $0<\alpha<1$. and $(\rho_0,u_0)\in H^{m+\alpha}(\T^d)\times H^{m+1}(\T^d)$, where $m>\frac d2+1$ and $\rho_0(x)>0$.
Let $T^*>0$ be the maximum existence time of the smooth solution for the uni-directional Euler-alignment system \eqref{eq:Euni} constructed in Theorem \ref{thm.local2}.
Then provided that
\begin{align}\label{eq:reg-cr}
  \sup_{t\in[0,T^{*})}\|u(t)\|_{ C^\sigma(\T^d)} < \infty,\quad \text{for some } \sigma\in(1-\alpha ,1),
\end{align}
we necessarily have $T^* = \infty$. Moreover, we obtain the following Lipschitz bounds:
\begin{align}\label{eq:rho-Lip0}
  \|\nabla \rho(t)\|_{L^\infty} \leq C \Big(1+ \|u\|_{L^\infty(\R_+;C^\sigma(\T^d))}^{\frac{1}{\sigma-1+\alpha}} \Big),\quad \forall t>0,
\end{align}
and
\begin{align}\label{eq:u-Lip0}
  \|\nabla u(t)\|_{L^\infty} \leq C \Big(1+ \|u\|_{L^\infty(\R_+;C^\sigma(\T^d))}^{\frac{1}{\sigma-1+\alpha}} \Big) e^{-c_0 t},
  \quad \forall t>0,
\end{align}
where $C>0$ depends only on $\alpha$, $d$, and initial data $(\rho_0,u_0)$;
and the rate $c_0 >0$ is the same as in Theorem \ref{thm.flocking}.
\end{theorem}

A regularity criterion has been established in \cite{lear2021global}, which is stated in \eqref{blow0} and asserts that solutions remain smooth if both $\rho$ and $u$ are Lipschitz continuous. In comparison, our regularity criterion \eqref{eq:reg-cr} imposes a less stringent condition by requiring only H\"older continuity of $u$. This represents a significant improvement in terms of the regularity requirement.

We would like to emphasize that the system \eqref{eq:Euni} exhibits an invariance property under the scaling transformation
\begin{equation}\label{eq:scaling}	
  \rho(x,t)\rightsquigarrow \rho(\lambda x,\lambda^\alpha t),\quad
  u(x,t)\rightsquigarrow \lambda^{-(1-\alpha)}u(\lambda x,\lambda^\alpha t), \quad \forall~\lambda >0.
\end{equation}
Consequently, the criterion \eqref{eq:reg-cr} only necessitates that $u$ belongs to a slightly smoother space compared to the scale-invariant class $L^\infty(\R_+; C^{1-\alpha}(\T^d))$. Our result shares similarities with the works of Constantin and Wu \cite{constantin2008regularity} on the supercritical quasi-geostrophic equation and Silvestre \cite{silvestre2012differentiability} on the advection-diffusion equation. We employ the same modulus of continuity method to obtain our result.
However, we have not attempted to extend our regularity criterion \eqref{eq:reg-cr} to the case of $u\in L^\infty(\R_+; C^{1-\alpha}(\T^d))$. If this were the case, one would expect that $\rho$ becomes H\"older continuous \cite{silvestre2012holder}. Further regularization of the solution is possible. See relevant discussion in Remark \ref{rmk:supcrit}.

It is worth noting that the regularity criterion \eqref{eq:reg-cr} is also expected to hold for the fractional Burgers equation \eqref{eq:fBurgers}, as it satisfies the same scaling \eqref{eq:scaling}. Moreover, in \cite{kiselev2008blow}, solutions were constructed in such a way that the regularity criterion fails in finite time, resulting in the development of singularities. However, it remains unclear whether such blow-up phenomena occur in the context of the uni-directional Euler-alignment system \eqref{eq:Euni}. This intriguing question will serve as the focus of future investigations.

The outline of our paper is as follows.
In Section \ref{sec:prel}, we present the local well-posedness result for system \eqref{eq:Euni} and establish some fundamental a priori bounds for the quantities $(\rho,u)$ and the auxiliary quantity $G$.
Our general approach revolves around the method of modulus of continuity (MOC). In Subsection \ref{subsec:MOC}, we set up a framework for simultaneously propagating the MOCs of $\rho$ and $u$, while also identifying potential breakdown scenarios that could violate their preservation. In Subsections \ref{subsec:GE-rho} and \ref{subsec:GE-u}, we demonstrate the general estimates for the evolution of the MOCs by density $\rho$ and velocity $u$, respectively, covering the entire range of $0<\alpha <2$ under possible breakdown scenarios.
Then, in Sections \ref{sec:thm-sub}, \ref{sec:crit}, and \ref{sec:reg-cr}, we respectively prove that the breakdown scenarios cannot occur in the subcritical ($1< \alpha <2$), critical ($\alpha=1$), and supercritical ($0<\alpha<1$) regimes. For the critical regime, we carefully select a pair of MOCs for $\rho$ and $u$ to avoid the occurrence of breakdown scenarios. The preservation of MOCs implies the uniform Lipschitz regularity of $(\rho,u)$, leading to the proofs of Theorems \ref{thm.main}, \ref{thm.flocking}, and \ref{thm:reg-cr}.
Finally, we provide the proofs of two auxiliary lemmas in the appendix section.

\vskip1mm
\noindent$\textbf{Notations}$:
For convenience, we sometimes use $\R^d$ instead of $\T^d$ by periodically extending the domain to the whole space.
The constant $C$  may be different from line to line, and the notation $a\lesssim b$ means $a\leq Cb$.

\section{Preliminaries}\label{sec:prel}
In this section, we state a collection of known results on the unidirectional Euler-alignment system \eqref{eq:Euni} in the existing literature. The 1D theory was established in \cite{do2018global,shvydkoy2017eulerian,shvydkoy2017eulerian2}, and the multi-dimensional case was discussed in \cite{lear2021unidirectional}.

\subsection{Local well-posedness}
We begin with the local well-posedness result for smooth solutions to the Euler-alignment system \eqref{eq:Euni}.
\begin{theorem}[Local well-posedness]\label{thm.local2}
Let $0<\alpha<2$. Suppose that $m>\frac{d}{2} + 1$ and
\begin{align*}
  (\rho_0, u_0)\in H^{m+\alpha}(\T^d)\times H^{m+1}(\T^d),
\end{align*}
with $\rho_0(x)>0$. Then there exists a $T_0>0$ such that the Euler-alignment system \eqref{eq:Euni} with initial data $(\rho_0,u_0)$ has a unique non-vacuous solution $(\rho,u)$ on interval $[0,T_0 )$ in the class
\begin{align*}
 \rho\in C_w([0,T_0); H^{m+\alpha}(\T^d)),\quad u\in C_w([0,T_0); H^{m+1}(\T^d))\cap L^2 ([0,T_0); \dot H^{m+1+\frac\alpha2}(\T^d)).
\end{align*}
Moreover, let $T^*>0$ be the maximal existence time of the above constructed solution, then
\begin{align}\label{blow0}
  \textrm{if}\;\; T^*<\infty,\quad \Longrightarrow \quad
  \sup_{t\in [0,T^*)}\big(\|\nabla \rho(t)\|_{L^\infty}+\|\nabla u(t)\|_{L^\infty}\big) = \infty.
\end{align}
\end{theorem}
The proof of the theorem can be found in e.g. \cite[Theorem 1.1]{lear2021unidirectional}. We omit the details.

Throughout the remainder of this paper, we will use the notation $T^*$ to represent the maximal existence time of the local smooth solution $(\rho,u)$ constructed in Theorem \ref{thm.local2} for the unidirectional Euler-alignment system \eqref{eq:Euni}. This notation will be consistently employed in our subsequent analysis.

\subsection{A priori bounds}\label{subsec:mp}
We list some useful a priori bounds on the solution $(\rho,u)$ and the auxiliary quantity $G := \partial_{x_1}u-\Lambda^\alpha \rho$.

First, by integrating the continuity equation $\eqref{eq:Euni}_1$ with respect to $x$-variable, we have the conservation of mass
\begin{align}\label{eq:mass-pres}
  \int_{\T^d} \rho(x,t) \dd x = \int_{\T^d} \rho_0 (x) \dd x =: \bar{\rho}_0.
\end{align}
We also have the conservation of momentum:
\begin{align*}
  \int_{\T^d} (\rho u)(x,t) \dd x = \int_{\T^d} (\rho_0 u_0)(x) \dd x,
\end{align*}
which can be deduced from the integration over $\T^d$ of the momentum equation
\begin{equation*}
  \partial_t(\rho u) + \partial_{x_1}(\rho u^2) = \rho\,\mathcal{C}_\alpha(u,\rho),
\end{equation*}
and use the fact
\begin{equation*}
  \int_{\T^d}\rho(x)\,\mathcal{C}_\alpha(u,\rho)(x)\dd x=\int_{\T^d}\int_{\T^d} \phi(x-y) (u(y)-u(x)) \rho(x)\rho(y) \dd x \dd y=0.
\end{equation*}

Next, we define $F := \frac{G}{\rho}$. Using the equation of $G$ in \eqref{eq:G}, we find
\begin{align}\label{eq.F}
  \partial_t F + u\,\partial_{x_1}F=0, \quad F|_{t=0}(x)=F_0(x).
\end{align}
It directly yields that
\begin{align}\label{es.F}
  \|F(t)\|_{L^\infty(\T^d)} \leq \|F_0\|_{L^\infty(\T^d)} = \Big\|\frac{\partial_{x_1}u_0 - \Lambda^\alpha \rho_0}{\rho_0} \Big\|_{L^\infty}.
\end{align}

From the relation $\partial_{x_1}u = G + \Lambda^\alpha \rho = F \rho + \Lambda^\alpha \rho$, we can write the continuity equation $\eqref{eq:Euni}_1$ as
\begin{align*}
  \partial_t\rho + u\,\partial_{x_1}\rho=- F \rho^2 - \rho\Lambda^\alpha\rho.
\end{align*}
This leads to the following a priori bounds on $\rho$.
\begin{proposition}\label{prop:MP-rho}
 There exist positive constants $\overline{\rho}$ and $\underline{\rho}$, depending on $\alpha, \bar{\rho}_0$ and $\|F_0\|_{L^\infty}$, such that
\begin{equation}\label{eq:rho-MP}
  0< \underline{\rho} \leq \rho(x,t)\leq \overline{\rho} <\infty,\qquad\forall~x\in\T^d,\,\,t\in[0,T^*).
\end{equation}
\end{proposition}
The upper bound can be obtained by using the nonlinear maximum principle introduced by Constantin and Vicol \cite{constantin2012nonlinear}. See e.g. \cite[Theorem 2.1]{do2018global} for applications to the 1D Euler-alignment system. A similar argument leads to a time-dependent lower bound $\rho(t)\gtrsim 1/t$. A uniform lower bound was first obtained in \cite{shvydkoy2017eulerian2}, making additional use of \eqref{eq:mass-pres}. We refer the detailed proof to \cite[Lemma 3.1]{shvydkoy2017eulerian2}.

In combination with \eqref{es.F}, we also get that for every $t\in [0,T^*)$,
\begin{align}\label{es:G}
  \|G(t)\|_{L^\infty(\T^d)} \leq \|F\rho(t)\|_{L^\infty(\T^d)}\leq \overline{\rho} \, \|F_0\|_{L^\infty(\T^d)}.
\end{align}

Finally, for the velocity $u$, let us recall
\begin{align}\label{eq:u}
  \partial_t u + u\, \partial_{x_1}u = c_\alpha\,\mathrm{p.v.} \int_{\R^d} \frac{u(y) - u(x)}{|x-y|^{d+\alpha}} \rho(y) \dd y.
\end{align}
The standard maximum principle yields the uniform bound
\begin{align*}
  \|u(t)\|_{L^\infty(\T^d)} \leq \|u_0\|_{L^\infty(\T^d)},\qquad \forall~ t\in[0,T^*).
\end{align*}
Moreover, we recall the following exponential decay estimate of $u$ (see \cite[Theorem 2.2]{tadmor2014critical}).
\begin{lemma}\label{lem:exp-decay}
  Let $\alpha\in (0,2)$. Assume that $u(x,t)$ is a smooth solution solving equation \eqref{eq:u} on $[0,T^*)$.
Denote by
\begin{align*}
  V(t) : = \sup_{x,y\in \mathrm{supp}\, \rho(\cdot ,t)} |u(x,t) - u(y,t)|.
\end{align*}
Then there exists a constant $c_0>0$ depending only on $\alpha,d$ such that for every
$t\in [0,T^*)$,
\begin{align}\label{eq:exp-decay}
  V(t) \leq V_0\, e^{-c_0 t}.
\end{align}
\end{lemma}



\section{General estimates on the evolution of the modulus of continuity}\label{sec:gwp}

Our primary analytical tool for studying the global well-posedness of the system is the innovative \emph{modulus of continuity method}. This method was initially developed by Kiselev et al. in \cite{kiselev2007global} for the critical quasi-geostrophic equation. It has proven effective in tackling various fluid equations with critical scalings and establishing global well-posedness results. Notably, the method has been successfully applied to the 1D Euler-alignment system in \cite{do2018global, kiselev2018global, miao2021global}, where global well-posedness is demonstrated for $0<\alpha<2$.

In this section, we establish the framework of the modulus of continuity method for our system \eqref{eq:Euni} and derive the necessary estimates to establish global well-posedness.

\subsection{The modulus of continuity}\label{subsec:MOC}
A function $\omega(\xi):(0,\infty)\rightarrow(0,\infty)$ is called a \emph{modulus of continuity} (MOC) if $\omega(\xi)$ is continuous, nondecreasing, concave, and piecewise $C^2$ with one-sided derivatives defined at every point in $(0,\infty)$.
We say a function $f$ obeys the modulus of continuity $\omega$ if
\begin{align*}
  |f(x)-f(y)|<\omega(|x-y|),\quad \mathrm{for\,\,all}\,\,x\neq y\in \R^d.
\end{align*}

We start with the following modulus of continuity
\begin{equation*}
\bar{\omega}^{\delta,\mu}(\xi): =
\left\{
\begin{array}{ll}
  \delta \big(\xi-\frac{1}{4}\xi^{1+ \mu}\big), & \quad \hbox{for \,\,$0<\xi\leq1$;} \\
  \frac{3}{4}\delta + \frac{\delta}{2}\log\xi, & \quad \hbox{for \,\,$\xi>1$,}
\end{array}
\right.
\end{equation*}
where $\mu\in (0, \min\{\alpha,1\})$ is fixed later and $\delta>0$ is a sufficiently small parameter to be chosen later.

Consider a family of MOC via scaling
\begin{equation}\label{def.omg}
\omega_\lambda^{\delta,\mu} (\xi) := \bar{\omega}^{\delta,\mu} (\tfrac{\xi}{\lambda})=
\left\{
\begin{array}{ll}
  \delta\lambda^{-1}\xi-\frac14\delta\lambda^{-1-\mu}\xi^{1+ \mu}, & \quad \hbox{for \,\,$0<\xi\leq\lambda$;} \\
  \frac{3}{4}\delta+ \frac{1}{2}\delta \log\frac\xi\lambda, & \quad \hbox{for \,\,$\xi>\lambda$}.
\end{array}
\right.
\end{equation}
The following lemma states that any bounded Lipschitz function obeys a MOC in this family. The proof can be found in e.g. \cite[Lemma 4.1]{miao2021global}.
\begin{lemma}\label{lem.data}
For any function $f\in W^{1,\infty}(\R^d)$ and for every $\lambda$ satisfying
\begin{align*}
  0<\lambda \leq \frac{2\|f\|_{L^\infty}}{\|\nabla f\|_{L^\infty}}e^{-4\delta^{-1}\|f\|_{L^\infty}},
\end{align*}
we have that $f$ obeys the MOC $\omega_\lambda^{\delta,\mu}$ defined in \eqref{def.omg}.
\end{lemma}
As $\rho_0$ and $u_0$ are Lipschitz functions, for any given parameters $\delta$ and $\mu$, we may pick a small enough $\lambda$ such that they both obey $\omega_\lambda^{\delta,\mu}$. We choose the following MOC for the density
\begin{align}\label{def:omeg1}
  \omega_1(\xi) := \omega_\lambda^{\delta_1,\mu}(\xi)
\end{align}
with some $0< \delta_1 < 1$. Our goal is to demonstrate that the density $\rho(t)$ obeys $\omega_1$ for all time. This result implies the desired Lipschitz bound:
\begin{align}\label{eq:rho-Lip}
  \|\nabla \rho(t)\|_{L^\infty} \leq \omega_1'(0^+)=\delta\lambda^{-1}<\infty,\quad \forall~ t\in [0,T^*).
\end{align}

As discussed in the introduction, our approach involves simultaneously propagating the MOCs on the density and velocity. For this purpose, we introduce the MOC on the velocity:
\begin{align}\label{def.omeg2}
  \omega_2(\xi) := \omega_\lambda^{\delta_2,\mu}(\xi),
\end{align}
with $0<\delta_2<1$, and our aim is to show that $u(t)$ satisfies the MOC $\omega_2$ for all time. In most cases, we can choose $\delta_1=\delta_2$, but we keep the flexibility of selecting different parameters $\delta_1$ and $\delta_2$. This flexibility will play a crucial role in the critical case when $\alpha=1$ (see Remark \ref{rmk:K}).

Furthermore, to obtain the decay estimate \eqref{eq:flock1}, we consider a time-dependent MOC on $u$:
\begin{align}\label{def.omeg2t}
  \omega_2(\xi,t) := e^{-c_0 \, t} \omega_2(\xi),
\end{align}
where $c_0>0$ is a constant appearing in Lemma \ref{lem:exp-decay}. If $u(t)$ satisfies the MOC $\omega_2(\xi,t)$, then we obtain
\begin{align}\label{eq:u-Lip}
  \|\nabla u(t)\|_{L^\infty}\leq \delta_2 \lambda^{-1}e^{- c_0 \,t}, \quad \forall t\in [0,T^*)
\end{align}
where the Lipschitz norm decays exponentially in time.


The following lemma characterizes the only possible \emph{breakthrough scenario} when the two MOCs are not satisfied simultaneously. We refer the reader to \cite{kiselev2007global,kiselev2011nonlocal} for the proof.

\begin{lemma}[Breakthrough scenarios]\label{lem.brk}
Let $\rho(x,t)$, $u(x,t)$ be smooth functions on $\T^d\times [0,T^*)$.
Assume that $\rho_0(x)$ and $u_0$ obeys the MOCs $\omega_1$ and $\omega_2$, defined in \eqref{def:omeg1} and \eqref{def.omeg2} respectively.
Let $t=t_1\in (0,T^*)$ be the first time that either $\rho(x,t)$ violates the MOC $\omega_1(\xi)$ given by \eqref{def:omeg1}
or $u(x,t)$ violates the MOC $\omega_2(\xi,t)$ given by \eqref{def.omeg2t}.
Then there exist two distinct points $x\neq y\in\T^d$ such that either
\begin{align}\label{break.mode0}
  \rho(x,t_1) - \rho(y,t_1) = \omega_1(\xi)\quad \mathrm{with}\,\,\xi=|x-y|,
\end{align}
or
\begin{align}\label{break.mode}
  u(x,t_1)-u(y,t_1)=\omega_2(\xi,t_1),
\end{align}
and also for any $\tilde{x},\tilde{y}\in \T^d$ and $t\in[0,t_1]$,
\begin{align}\label{break.mode1}
  |\rho(\tilde{x},t) -\rho(\tilde{y},t)|\leq \omega_1(|\tilde{x}-\tilde{y}|),
  \quad |u(\tilde{x},t)-u(\tilde{y},t)|\leq \omega_2(|\tilde{x}-\tilde{y}|,t).
\end{align}
\end{lemma}

Hence, in order to show that for all time $t\in (0,T^*)$ the solution $\rho(x,t)$ obeys the MOC $\omega_1(\xi)$
and simultaneously $u(x,t)$ obeys the MOC $\omega_2(\xi,t)$, we only need to consider two cases:
\begin{enumerate}[(i)]
\item No breakthrough for the MOC of $\rho$: under the scenario \eqref{break.mode0}, \eqref{break.mode1}, it suffices to show that
\begin{align}\label{aim.1}
  \partial_t \big(\rho(x,t) -\rho (y,t) \big)\big|_{t=t_1} <0;
\end{align}
\item No breakthrough for the MOC of $u$: under the scenario \eqref{break.mode}-\eqref{break.mode1}, it suffices to show that
\begin{align*}
  \partial_t \bigg( \frac{u(x,t)-u(y,t)}{\omega_2(\xi,t)}\bigg)\bigg|_{t=t_1}<0,
\end{align*}
or equivalently,
\begin{align}\label{aim-equi}
  \partial_t\big(u(x,t) -u(y,t) \big)|_{t=t_1} + c_0\, \omega_2(\xi,t_1) <0.
\end{align}
\end{enumerate}
If these estimates \eqref{aim.1}--\eqref{aim-equi} are proven, it leads to a contradiction and thus implies that the breakthrough scenario in Lemma \ref{lem.brk} cannot occur at any time.

Let us provide further comments on the two cases mentioned above. In the case (i), if $\xi>\omega_1^{-1}(\overline{\rho})$, we recall \eqref{eq:rho-MP} and find that
\[\rho(x,t_1) - \rho(y,t_1) \leq \overline{\rho}<\omega_1(\xi).\]
Therefore, scenario \eqref{break.mode0} cannot occur. Thus, we only need to establish \eqref{aim.1} for
\begin{align}\label{def:Xi1}
  0<\xi \leq \Xi_1 := \omega_1^{-1}(\overline{\rho}) =\lambda e^{2\delta_1^{-1} \overline{\rho} -\frac{3}{2}}.
\end{align}

Similarly, in the case (ii), if $\xi>\omega_2^{-1}(V_0)$, we recall \eqref{eq:exp-decay} and find that
\[u(x,t_1) - u(y,t_1) \leq V(t_1)\leq V_0\, e^{-c_0 t_1}<\omega_2(\xi)e^{-c_0 t_1}=\omega_2(\xi,t).\]
Therefore, scenario \eqref{break.mode} cannot occur. Thus, we only need to establish \eqref{aim-equi} for
\begin{align}\label{def:Xi2}
  0<\xi\leq \Xi_2:= \omega_2^{-1}(V_0) = \lambda e^{2\delta_2^{-1}V_0 -\frac{3}{2}}.
\end{align}

We may further choose $\lambda$ to be sufficiently small as 
\begin{align*}
  \lambda\leq \tfrac{1}{2} e^{- \big(2\delta_1^{-1} \overline{\rho} + 2\delta_2^{-1}V_0\big)},
\end{align*}
to ensure $\Xi_1,\Xi_2\leq \frac12$. 

Before we proceed, let us introduce some notational conventions for the sake of convenience. Since there are several quantities related to both $\rho$ and $u$ that have similar expressions, we will use a subscript $i$ to denote the common representation. Specifically, we will use $i=1$ and $i=2$ to refer to the quantities related to $\rho$ and $u$, respectively.

\subsection{Evolution of the MOC on $\rho$}\label{subsec:GE-rho}
We begin by presenting general estimates that lead to \eqref{aim.1} under the scenario \eqref{break.mode0}, \eqref{break.mode1}. The analysis for the 1D Euler-alignment system has been conducted in \cite{do2018global}, and we follow a similar procedure. However, it is important to note that additional difficulties arise due to the higher dimension $d>1$.

Below we drop the dependence on the variable $t_1$ for simplicity.
Taking advantage of the equation $\eqref{eq:Euni}_1$ and the relations $\partial_{x_1}u = \Lambda^\alpha\rho+G$,
$G = F\rho$ (recalling \eqref{eq:G} and \eqref{eq.F}), we see that
\[
  \partial_t \rho = - \partial_{x_1} (u\,\rho) = - \rho \Lambda^\alpha \rho - \rho^2 F - u\,\partial_{x_1} \rho,
\]
and thus
\begin{align}\label{Lip.rho}
  \partial_t \rho(x)- \partial_t \rho(y)
  & = -\rho(y)\big(\Lambda^\alpha \rho(x)-\Lambda^\alpha\rho(y)\big)
  -\big(\rho(x)-\rho(y)\big)\partial_{x_1}u(x)\nonumber\\
  &\quad
  -\rho(y)F(x)\big(\rho(x)-\rho(y)\big) - \rho^2(y)\big(F(x)-F(y)\big)
  -\big((u\partial_{x_1}\rho)(x)-(u\partial_{x_1}\rho)(y)\big)\nonumber\\
  & =: N_1 + N_2 + N_3 + N_4 + N_5.
\end{align}

The first term $N_1$ in the estimate encodes the dissipation. Indeed, along the lines of \cite{kiselev2007global,kiselev2011nonlocal}, we have
\begin{equation}\label{eq:D}
  \Lambda^\alpha \rho(x)-\Lambda^\alpha \rho(y) = c_\alpha\,
  \mathrm{p.v.} \int_{\R^d} \frac{\omega_1(\xi) - \rho(x+z) + \rho(y+z)}{|z|^{d+\alpha}} \mathrm{d}z \geq D_{\alpha,1}(\xi)> 0,
\end{equation}
where we denote
\begin{equation}\label{es:Dalp}
\begin{split}
  D_{\alpha,i}(\xi)
  & := C_1 \bigg(\int_0^{\frac {\xi}{2}}
  \frac{2\omega_i(\xi)-\omega_i(\xi+2\eta) - \omega_i(\xi-2\eta)}{\eta^{1+\alpha}} \mathrm{d}\eta \\
  & \qquad\qquad + \int_{\frac{\xi}{2}}^\infty \frac{2\omega_i(\xi)- \omega_i(2\eta+\xi) + \omega_i(2\eta-\xi)}{\eta^{1+\alpha}}
  \mathrm{d}\eta \bigg)
\end{split}
\end{equation}
with the constant $C_1>0$ that depends only on $\alpha$ and $d$. Note that $D_{\alpha,i}(\xi)$ is strictly positive as $\omega_i$ is concave.
Clearly, \eqref{eq:D} implies
\begin{equation}\label{es:N3}
  N_1 \leq - \underline{\rho}\, D_{\alpha,1}(\xi).
\end{equation}

Next, for the term $N_2$, it follows from $\partial_{x_1}u = \Lambda^\alpha \rho + \rho F$ and \eqref{es.F} that
\begin{equation}\label{es:N2}
  N_2 \leq - \omega_1(\xi) \Lambda^\alpha \rho(x) +\overline{\rho}\, \|F_0\|_{L^\infty} \omega_1(\xi).
\end{equation}
Following \cite{do2018global}, we obtain the following bound on $-\Lambda^\alpha \rho(x)$,
\begin{equation}\label{es:Lam-al-rho}
  -\Lambda^\alpha \rho (x) = \, c_\alpha\,\mathrm{p.v.} \int_{\R^d} \frac{\big( \rho(x-z)-\rho(y)\big) - \big(\rho(x) - \rho(y)\big)}{|z|^{d+\alpha}} \mathrm{d}z \leq A_{\alpha,1}(\xi),
\end{equation}
where
\begin{equation}\label{def:A-alp}
  A_{\alpha,i}(\xi) : = c_\alpha\,
  \mathrm{p.v.} \int_{\R^d} \frac{ \omega_i(|\xi e_1 -z|) - \omega_i(\xi)}{|z|^{d+\alpha}} \mathrm{d}z,
\end{equation}
and $e_1 := (1,0,\dots,0)$.
Here, due to the rotation and translation invariance, we may assume without loss of generality that
\[
  x=\big(\tfrac{\xi}{2},0,\dots,0\big), \quad y=\big(-\tfrac{\xi}{2},0,\dots,0\big),
\]
so that $x-y = \xi e_1$.

The term $N_3$ can be easily controlled by
\begin{align}\label{es:N4}
  |N_3| \leq \overline{\rho}\, \|F_0\|_{L^\infty} \omega_1(\xi),
\end{align}
using \eqref{eq:rho-MP} and \eqref{es.F}.

For the term $N_4$, we have
\[|F(x)-F(y)|\leq\|\nabla F\|_{L^\infty}\xi.\]
To control $\partial_{x_1}F$, we introduce $H : = \frac{\partial_{x_1} F}{\rho}$, which satisfies
\[
  \partial_t H + u \partial_{x_1} H =0,\quad\text{with}\quad H_0 = \tfrac{\partial_{x_1} F_0}{\rho_0} = \tfrac{1}{\rho_0} \partial_{x_1}
  \Big(\tfrac{\partial_{x_1} u_0 - \Lambda^\alpha \rho_0}{\rho_0}\Big),
\]
which immediately implies, for every $t\in [0,T^*)$,
\[
  \|H(t)\|_{L^\infty} \leq \|H_0\|_{L^\infty},\quad\text{and}\quad
  \|\partial_{x_1} F(t)\|_{L^\infty} \leq \overline{\rho}\, \|H_0\|_{L^\infty}.
\]
For $d \geq 2$, additional control on the full gradient $\nabla F$ is required. To obtain this control, we use \eqref{eq.F} and compute
\[
  \partial_t \nabla F + u\,\partial_{x_1} \nabla F = - \nabla u\, \partial_{x_1}F, \quad \text{with}\quad\nabla F_0 = \nabla \Big(\tfrac{\partial_{x_1}u_0 -\Lambda^\alpha \rho_0}{\rho_0}\Big).
\]
This yields
\begin{align*}
  \|\nabla F(t)\|_{L^\infty} & \leq \|\nabla F_0\|_{L^\infty} + \int_0^t \|\nabla u(\tau)\|_{L^\infty}
  \|\partial_{x_1} F(\tau)\|_{L^\infty} \dd\tau \\
  & \leq \|\nabla F_0\|_{L^\infty} + \overline{\rho}\, \|H_0\|_{L^\infty} \int_0^t \|\nabla u(\tau)\|_{L^\infty} \dd \tau.
\end{align*}
Given the scenario \eqref{break.mode1}, $u(t)$ satisfies $\omega_2(\xi,t)$ as defined in \eqref{def.omeg2t}. Thus,
\begin{align*}
  \|\nabla u(t)\|_{L^\infty} \leq e^{- c_0 \, t} \omega_2'(0^+) = e^{-c_0\,t} \delta_2 \lambda^{-1}, \quad \forall\, t\in [0,t_1].
\end{align*}
By integrating over time, we obtain
\[
\int_0^{t_1} \|\nabla u(\tau)\|_{L^\infty} \dd \tau \leq \delta_2\lambda^{-1}\int_0^{t_1}e^{- c_0\,t}\dd t\leq \frac{\delta_2}{c_0\lambda}.
\]
Hence the term $N_4$ can be estimated as follows
\begin{align}\label{es:N5}
  |N_4| \leq \overline{\rho}^2 \|\nabla F(t_1)\|_{L^\infty} \xi \leq \overline{\rho}^2
  \Big(\|\nabla F_0\|_{L^\infty}  + \tfrac{\overline{\rho}}{c_0} \|H_0\|_{L^\infty} \delta_2 \lambda^{-1}\Big) \xi.
\end{align}

Finally, for the advection term $N_5$, we find (e.g. see \cite{kiselev2007global})
\begin{equation}\label{es:N1}
  |N_5| \leq |u(x) - u(y)|\, \omega_1'(\xi).
\end{equation}
\begin{remark}\label{rmk:drift}
In one dimension, one can take advantage of the relation
\[
u=\partial_x^{-1}(\Lambda^\alpha\rho+G)=-\partial_x\Lambda^{\alpha-2}\rho+\partial_x^{-1}G,
\]
and use $\omega_1$ to control the MOC of $u$ (see \cite[Lemma 4.4]{do2018global}). However, in higher dimensions, we cannot expect that the MOC of $u$ can be controlled by $\omega_1$ since the relation only involves the partial derivative of $u$ in the $e_1$ direction. Therefore, we will separately show that $u(t)$ obeys $\omega_2(\xi,t)$ as defined in \eqref{def.omeg2t}. It is worth noting that when $\alpha\in(0,1)$, the term $\omega_2(\xi,t)\omega_1'(\xi)$ cannot be controlled by the dissipation. We need additional assistance from the regularity condition \eqref{eq:reg-cr}. Detailed calculations to establish these estimates will be provided in the subsequent sections.	
\end{remark}

Combining the estimates \eqref{Lip.rho}, \eqref{es:N3}, \eqref{es:N2}, \eqref{es:Lam-al-rho}, \eqref{es:N4}, \eqref{es:N5}, \eqref{es:N1}, we can deduce that for every $0 < \alpha < 2$ and $\xi > 0$,
\begin{equation}\label{eq:GE-rho}
\begin{split}
  &\partial_t\rho(x,t_1)- \partial_t\rho(y,t_1)
  \leq -\underline{\rho}\, D_{\alpha,1}(\xi)+\omega_1(\xi) \big(A_{\alpha,1}(\xi)+ \overline{\rho} \|F_0\|_{L^\infty}\big)\\
  &\qquad+\overline{\rho}^2
  \Big(\|\nabla F_0\|_{L^\infty}  + \tfrac{\overline{\rho}}{c_0} \|H_0\|_{L^\infty} \delta_2 \lambda^{-1}\Big) \xi +|u(x) - u(y)|\, \omega_1'(\xi).
\end{split}
\end{equation}

Now we further estimate the terms on the right hand side of \eqref{eq:GE-rho}. The goal is to use the first term to control all the rest. We start with a lower bound on the dissipative term $D_{\alpha,1}$.

\begin{lemma}[Dissipation bound]\label{lem:D}
Let $\omega_i(\xi)$ be the modulus of continuity given by \eqref{def:omeg1} or \eqref{def.omeg2}. Then for every $\alpha\in (0,2)$ and for any $\xi>0$, we have
\begin{equation}\label{es:Dalp2}
  D_{\alpha,i}(\xi)\geq
\begin{cases}
  \frac{C_1\mu(\mu+1)2^{\alpha-1} }{4(2-\alpha)} \delta_i \lambda^{-1-\mu}\xi^{1+\mu-\alpha}, & \quad \mathrm{for} \,\,\;0<\xi\leq\lambda, \\
  \frac{C_1 2^{\alpha-1}}{\alpha} \omega_i(\xi)\xi^{-\alpha}, & \quad \mathrm{for} \,\,\;\xi>\lambda.
\end{cases}
\end{equation}
\end{lemma}
The dissipation bound was originally derived in \cite{kiselev2007global}. We include a proof under our notations in the Appendix for self-consistency.

The next lemma provides a bound on the term $A_{\alpha,1}$. In the case when $d=1$, this bound was derived in \cite[Lemma 4.5]{do2018global}. However, in the multi-dimensional case, a significant enhancement is required specifically for the directions orthogonal to $e_1$.

\begin{lemma}\label{lem:Aalp}
Let $\omega_i(\xi)$ be the modulus of continuity given by \eqref{def:omeg1} or \eqref{def.omeg2}. Then for every $\alpha\in (0,2)$ and for any $\xi>0$, we have
\begin{equation}\label{es:Aalp}
  A_{\alpha,i}(\xi) \leq
  \begin{cases}
    C_2 \delta_i \lambda^{-\mu} \xi^{\mu - \alpha},\quad & \mathrm{for}\;\; 0<\xi \leq \lambda, \\
    C_2 \delta_i \xi^{-\alpha}, \quad & \mathrm{for}\;\; \xi \geq \lambda,
  \end{cases}
\end{equation}
where $C_2>0$ depends only on $\alpha$, $d$ and $\mu$.
\end{lemma}

\begin{proof}[Proof of Lemma \ref{lem:Aalp}]


Let us denote $z=(z_1,z_h)$ with $z_h =(z_2,\cdots,z_d)$.
We split $A_{\alpha,i}(\xi)$ given by \eqref{def:A-alp} as follows
\begin{align}\label{split:Aalp}
  A_{\alpha,i}(\xi)
  & = c_\alpha \,\mathrm{p.v.} \int_{|z_1|\leq 2\xi} \int_{\R^{d-1}}
  \frac{ \omega_i(|\xi -z_1|) -\omega_i(\xi) }{|z|^{d+\alpha}} \dd z_h \dd z_1 \nonumber \\
  & \quad + c_\alpha \, \mathrm{p.v.} \int_{|z_1|\leq 2\xi} \int_{\R^{d-1}}
  \frac{\omega_i(|\xi e_1 -z|) - \omega_i(|\xi-z_1|) }{|z|^{d+\alpha}} \dd z_h \dd z_1 \nonumber \\
  & \quad + c_\alpha \,\mathrm{p.v.} \int_{|z_1|\geq 2\xi} \int_{\R^{d-1}}
  \frac{- \omega_i(\xi) + \omega_i(|\xi e_1 -z|)  }{|z|^{d+\alpha}} \dd z_h \dd z_1 \nonumber \\
  & = : I_{i,1} + I_{i,2} + I_{i,3}.
\end{align}
For $I_{i,1}$, using symmetry, we get
\begin{align}
  I_{i,1} = & \,c_\alpha\,\mathrm{p.v.} \int_0^\xi \int_{\R^{d-1}}
  \frac{-2\omega_i(\xi) + \omega_i(\xi-z_1) + \omega_i(\xi+z_1) }{|z|^{d+\alpha }}
  \mathrm{d}z_h \mathrm{d}z_1  \nonumber\\
  & + c_\alpha\, \mathrm{p.v.}\int_\xi^{2\xi}\int_{\R^{d-1}}
  \frac{ \omega_i(\xi-z_1) -\omega_i(\xi)}{|z|^{d+\alpha }}
  \mathrm{d}z_h \mathrm{d}z_1 \nonumber \\
  & + c_\alpha \,\mathrm{p.v.}\int_\xi^{2\xi}\int_{\R^{d-1}}
  \frac{\omega_i(\xi+z_1) - \omega_i(\xi)}{|z|^{d+\alpha}} \mathrm{d}z_h \mathrm{d}z_1,\label{eq:Ii1}
\end{align}
and the first two integrals on the right-hand side of the above formula are both negative due to the concavity of $\omega_i$ ($i=1,2$),
which gives
\begin{align*}
  I_{i,1} \leq c_\alpha \,\int_\xi^{2\xi} \int_{\R^{d-1}} \frac{\omega_i(\xi +z_1) -\omega_i(\xi)}{|z|^{d+\alpha}}
  \dd z_h \dd z_1
  \leq c_\alpha C_{d,\alpha} \int_\xi^{2\xi} \frac{\omega_i(\xi +z_1) -\omega_i(\xi)}{z_1^{1+\alpha}} \dd z_1 ,
\end{align*}
with $C_{d,\alpha} = \int_{\R^{d-1}} \frac{1}{(1+ |z_h|^2)^{\frac{d+\alpha}{2}}} \dd z_h <\infty$.
Arguing as the the estimation in \cite[Lemma 4.4]{miao2021global} (with $\gamma = \frac{\delta_i}{2}$), we get
\begin{equation}\label{es:Aalp-1}
  I_{i,1} \leq
  \begin{cases}
    c_\alpha C_{d,\alpha} \delta_i \overline{M}_\alpha(\xi,\lambda),
    \quad & \textrm{for}\;\; 0<\xi \leq \lambda, \\
    \frac{ c_\alpha C_{d,\alpha}}{\alpha} \delta_i \xi^{-\alpha},\quad & \textrm{for}\;\;  \xi \geq \lambda ,
  \end{cases}
\end{equation}
where
\begin{equation*}
\overline{M}_\alpha(\xi,\lambda) : =
\begin{cases}
  \frac{1}{\alpha^2 (1-\alpha)} \lambda^{-\alpha},\quad & \textrm{for}\;\; 0<\alpha <1, \\
  \lambda^{-1}\big(\log \frac{\lambda}{\xi} + \frac{5}{4} \big), \quad & \textrm{for}\;\; \alpha=1, \\
  \big(\frac{1}{\alpha-1} + \frac{5}{4}\big) \lambda^{-1} \xi^{1-\alpha}, \quad & \textrm{for}\;\; 1<\alpha <2.
\end{cases}
\end{equation*}
In order to compare with the dissipation contribution, we state the following inequality,
where we only use the fact that $\frac{\xi}{\lambda} \in (0,1]$ and $\mu\in (0,\min\{1,\alpha\})$,
\begin{equation}\label{eq:overC-alp}
  \overline{M}_\alpha(\xi,\lambda)
  \leq \overline{C}_{\alpha,\mu} \lambda^{-\mu} \xi^{\mu-\alpha},\quad
  \textrm{with}\;\;
  \overline{C}_{\alpha,\mu}
  : = \begin{cases}
    \frac{1}{\alpha^2(1-\alpha)},\quad & \textrm{for}\;\;0<\alpha<1, \\
    \frac{5}{4} + \sup\limits_{r\geq 1}\frac{\log r}{r^{1-\mu}},\quad & \textrm{for}\;\; \alpha=1 ,\\
    \frac{1}{\alpha-1} + \frac{5}{4}, \quad & \textrm{for}\;\; 1<\alpha <2.
  \end{cases}
\end{equation}

For $I_{i,2}$ given by \eqref{split:Aalp}, we separately consider two cases: for every $0<\xi \leq 2\lambda$,
noting that
\begin{align}\label{eq:fact0}
  \sqrt{(\xi-z_1)^2+|z_h|^2}-|\xi-z_1| = \tfrac{|z_h|^2}{ \sqrt{|\xi-z_1|^2+|z_h|^2}+ |\xi-z_1|},
\end{align}
and using the fact that $\omega_i'(\eta)\leq \omega_i'(0^+)= \delta_i \lambda^{-1}$ for all $\eta\in \R_+$, we get
\begin{align}\label{es.Aalp-2low}
  I_{i,2}
  & \leq c_\alpha \,\mathrm{p.v.} \int_{|z_1|\leq 2\xi} \int_{\R^{d-1}} \frac{\delta_i \lambda^{-1}}{|z|^{d+\alpha}}
  \frac{|z_h|^2}{ \sqrt{|\xi-z_1|^2+|z_h|^2}+|\xi-z_1|} \mathrm{d}z_h\mathrm{d}z_1 \nonumber\\
  & \leq c_\alpha  \,\mathrm{p.v.}\int_{|z_1|\leq\frac{\xi}{2}}\int_{|z_h|\leq\xi}
  \frac{\delta_i \lambda^{-1}}{|z|^{d+\alpha}} \frac {|z_h|^2}{\xi/2}
  \mathrm{d}z_h \mathrm{d} z_1 +  c_\alpha \,\mathrm{p.v.} \int_{|z_1|\leq 2\xi}\int_{|z_h|\ge\xi}
  \frac{\delta_i \lambda^{-1}}{|z|^{d+\alpha }} |z_h|
  \mathrm{d}z_h\mathrm{d}z_1 \nonumber\\
  & \quad + c_\alpha \,\mathrm{p.v.} \int_{\frac{\xi}{2} \leq |z_1| \leq 2\xi} \int_{|z_h|\leq\xi}
  \frac{\delta_i \lambda^{-1} }{|z|^{d+\alpha}} |z_h| \mathrm{d} z_h\mathrm{d}z_1 \nonumber\\
  &\leq C_0 c_\alpha \delta_i \lambda^{-1} \bigg( \xi^{-1}\int_{|z|\leq 3\xi} \frac{1}{|z|^{d-2 +\alpha}} \dd z
  +  \xi \int_{|z_h|\geq \xi} \frac{1}{|z_h|^{d-1+\alpha}} \dd z_h
  +  \sigma_{d-2} \xi^{1-\alpha} \bigg) \nonumber \\
  & \leq C_0 c_\alpha \frac{\sigma_{d-1} +\sigma_{d-2}}{\alpha(2-\alpha)} \delta_i \lambda^{-1} \xi^{1-\alpha},
\end{align}
with $\sigma_n$ denoting the area of $n$-dimensional sphere for $n\ge1$ (setting $\sigma_0 = 1$);
whereas for every $ \xi\geq 2\lambda$, we have
\begin{align}\label{split:Aalp-12}
  I_{i,2} & = c_\alpha \,\mathrm{p.v.}
  \int_{|z_1|\leq 2\xi} \int_{|z_h|\geq \xi} \frac{\omega_i(|(\xi-z_1,z_h)|) -\omega_i(|\xi-z_1|)}{|z|^{d+\alpha}}
  \dd z_h \dd z_1 \nonumber\\
  & \quad + c_\alpha \, \mathrm{p.v.}
  \int_{|z_1|\leq \frac{\xi}{2}} \int_{|z_h|\leq \xi} \frac{\omega_i(|(\xi-z_1,z_h)|)
  -\omega_i(|\xi-z_1|)}{|z|^{d+\alpha}}  \dd z_h \dd z_1 \nonumber \\
  & \quad +  c_\alpha \,\mathrm{p.v.}
  \int_{\{\frac{\xi}{2}\leq |z_1|\leq 2\xi\}\cap \{|z_1 -\xi|\geq \lambda\}} \int_{|z_h|\leq \xi} \frac{\omega_i(|(\xi-z_1,z_h)|) -\omega_i(|\xi-z_1|)}{|z|^{d+\alpha}}
  \dd z_h \dd z_1 \nonumber\\
  & \quad + c_\alpha \, \mathrm{p.v.}
  \int_{|z_1-\xi|\leq \lambda} \int_{|z_h|\leq \xi} \frac{\omega_i(|(\xi-z_1,z_h)|)
  -\omega_i(|\xi-z_1|)}{|z|^{d+\alpha}}  \dd z_h \dd z_1 \nonumber \\
  & = : I_{i,2,1} + I_{i,2,2} + I_{i,2,3} + I_{i,2,4}.
\end{align}
By using \eqref{def.omg}, \eqref{eq:fact0} and the following fact that
$\sup_{r\in [1/2,\infty)} r^{-\frac{1}{2}} \log\sqrt{1+r^2} \leq C_0$,
we find
\begin{align}\label{es:Aalp-21}
  I_{i,2,1} & = c_\alpha \,\mathrm{p.v.} \int_{\{|z_1|\leq 2\xi\}\cap \{|z_1-\xi|\geq \lambda\}}
  \int_{|z_h|\geq \xi} \frac{\omega_i(|(\xi-z_1,z_h)|) -\omega_i(|\xi-z_1|)}{|z|^{d+\alpha}}\dd z_h \dd z_1 \nonumber \\
  &\quad + c_\alpha \,\mathrm{p.v.} \int_{|z_1-\xi|\leq \lambda}
  \int_{|z_h|\geq \xi} \frac{\omega_i(|(\xi-z_1,z_h)|) -\omega_i(|\xi-z_1|)}{|z|^{d+\alpha}}\dd z_h \dd z_1 \nonumber \\
  & \leq c_\alpha \,\mathrm{p.v.} \int_{|z_1|\leq 2\xi} \int_{|z_h|\geq \xi} \frac{\frac{\delta_i}{2}
  \log \sqrt{1+ \frac{|z_h|^2}{|\xi-z_1|^2}}}{|z|^{d+\alpha}}   \dd z_h \dd z_1 \nonumber \\
  & \quad + c_\alpha \,\mathrm{p.v.} \int_{|z_1-\xi|\leq \lambda} \int_{|z_h|\geq \xi}
  \frac{\delta_i \lambda^{-1}}{2|z|^{d+\alpha}}
  \frac{|z_h|^2}{ \sqrt{|\xi-z_1|^2+|z_h|^2}+ |\xi-z_1|} \dd z_h \dd z_1 \nonumber \\
  & \leq C_0 c_\alpha \delta_i \int_{|z_1|\leq 2\xi} \int_{|z_h|\geq \xi} \frac{1}{|z_h|^{d+\alpha}}
  \frac{|z_h|^{1/2}}{|\xi-z_1|^{1/2}} \dd z_h \dd z_1
  + C_0 c_\alpha \delta_i \int_{|z_h|\geq \xi} \frac{1}{|z_h|^{d+\alpha}} |z_h| \dd z_h \nonumber \\
  & \leq C_0 c_\alpha \sigma_{d-2} \delta_i \bigg(\xi^{-\frac{1}{2} -\alpha} \int_{|z_1|\leq 2\xi} |\xi-z_1|^{-\frac{1}{2}} \dd z_1
  + \frac{1}{\alpha} \xi^{-\alpha}\bigg)
  \leq \frac{C_0 c_\alpha \sigma_{d-2}}{\alpha} \delta_i \xi^{-\alpha}.
\end{align}
By virtue of \eqref{def.omg}, \eqref{eq:fact0} and the fact that
\begin{align*}
  \log |(\eta,z_h)| -\log |\eta| \leq |\eta|^{-1} \big(|(\eta,z_h)| - |\eta| \big),\quad  \forall |\eta|>0,z_h\in\R^{d-1},
\end{align*}
we estimate $I_{i,2,2}$ as follows that for every $\xi \geq 2\lambda$,
\begin{align*}
  I_{i,2,2} & = c_\alpha \,\mathrm{p.v.} \int_{|z_1|\leq \frac{\xi}{2}}
  \int_{|z_h|\leq \xi}
  \frac{\delta_i \big(\log |(\xi-z_1,z_h)| - \log |\xi-z_1|\big)}{2|z|^{d+\alpha}} \dd z_h \dd z_1
  \nonumber \\
  & \leq c_\alpha \,\mathrm{p.v.} \int_{|z_1|\leq \frac{\xi}{2}} \int_{|z_h|\leq \xi}
  \frac{\delta_i |\xi-z_1|^{-1}}{2|z|^{d+\alpha}}
  \frac{|z_h|^2}{ \sqrt{|\xi-z_1|^2+|z_h|^2}+ |\xi-z_1|} \dd z_h \dd z_1 \nonumber \\
  & \leq C_0 c_\alpha \delta_i \xi^{-2} \int_{|z|\leq 2\xi} \frac{1}{|z|^{d-2+\alpha}} \dd z
  \leq \frac{C_0 c_\alpha \sigma_{d-1}}{2-\alpha} \delta_i \xi^{-\alpha}.
\end{align*}
Arguing as \eqref{es:Aalp-21} gives
\begin{align*}
  I_{i,2,3} & = c_\alpha \,\mathrm{p.v.} \int_{\{\frac{\xi}{2}\leq |z_1| \leq 2\xi\}\cap \{|z_1-\xi|\geq \lambda\}}
  \int_{|z_h|\leq \xi}
  \frac{\delta_i \big(\log |(\xi-z_1,z_h)| - \log |\xi-z_1|\big)}{2|z|^{d+\alpha}} \dd z_h \dd z_1
  \nonumber \\
  & \leq c_\alpha \int_{\frac{\xi}{2}\leq|z_1|\leq 2\xi} \int_{|z_h|\leq \xi}
  \frac{\delta_i \log \sqrt{1 + \frac{\xi^2}{|\xi-z_1|^2}}}{2(\xi/2)^{d+\alpha}} \dd z_h \dd z_1 \nonumber \\
  & \leq C_0 c_\alpha \sigma_{d-2} \delta_i \xi^{-1-\alpha}  \int_{\frac{\xi}{2}\leq |z_1|\leq 2\xi }
  \frac{\xi^{1/2}}{|\xi-z_1|^{1/2}} \dd z_1 \leq C_0 c_\alpha \sigma_{d-2} \delta_i \xi^{-\alpha};
\end{align*}
Note that $|z|\geq |z_1| \geq \frac{\xi}{2}$ for every $ \xi \geq 2\lambda$ and $|z_1-\xi|\leq \lambda$, we directly have
\begin{align*}
  I_{i,2,4} &  \leq c_\alpha\, \mathrm{p.v.} \int_{|z_1-\xi|\leq \lambda} \int_{|z_h|\leq \xi}
  \frac{\delta_i \lambda^{-1}}{ |z|^{d+\alpha}} \frac{|z_h|^2}{ \sqrt{|\xi-z_1|^2+|z_h|^2}+ |\xi-z_1|} \dd z_h \dd z_1
  \nonumber \\
  & \leq c_\alpha \int_{|z_1-\xi|\leq \lambda} \int_{|z_h|\leq \xi}
  \frac{\delta_i \lambda^{-1}}{(\xi/2)^{d+\alpha}} |z_h| \dd z_h \dd z_1 \leq C_0 2^d c_\alpha \sigma_{d-2} \delta_i \xi^{-\alpha}.
\end{align*}
Gathering \eqref{es.Aalp-2low}, \eqref{split:Aalp-12} and the above estimates on $I_{i,2,1}$ - $I_{i,2,4}$ leads to that
\begin{equation}\label{es:Aalp-2}
  I_{i,2} \leq
  \begin{cases}
    C \delta_i \lambda^{-1} \xi^{1-\alpha},\quad & \textrm{for}\;\; 0<\xi \leq \lambda, \\
    C \delta_i \xi^{-\alpha}, \quad & \textrm{for}\;\; \xi\geq \lambda,
  \end{cases}
\end{equation}
with $C>0$ depending only on $\alpha,d$.

For $I_{i,3}$ given by \eqref{split:Aalp}, we in fact can control a larger quantity $\widetilde{I}_{i,3}$ given by
\begin{align}\label{def:tild-Ii3}
  \widetilde{I}_{i,3} : = c_\alpha \int_{ |z|\geq 2\xi}
  \frac{\omega_i(|\xi e_1 -z|) -\omega_i(\xi) }{|z|^{d+\alpha}} \dd z.
\end{align}
Noting that from concavity $\omega_i(|\xi e_1-z|) -\omega_i(\xi) \leq \omega_i(\xi+|z|)
-\omega_i(\xi) \leq \omega_i(|z|)$, and exactly arguing as \cite[(5.8)]{miao2021global} and \eqref{eq:overC-alp},
we have that for every $0<\xi \leq \lambda$,
\begin{align*}
  I_{i,3} \leq \widetilde{I}_{i,3} 
  \leq c_\alpha \, \mathrm{p.v.} \int_{|z|\geq 2 \xi} \frac{\omega_i(|z|)}{|z|^{d+\alpha}} \dd z
  \leq c_\alpha \, \sigma_{d-1} \int_\xi^\infty \frac{\omega_i(\eta)}{\eta^{1+\alpha}} \dd \eta
  \leq c_\alpha \sigma_{d-1} \overline{C}_{\alpha,\mu} \delta_i \lambda^{- \mu}
  \xi^{\mu- \alpha},
\end{align*}
with $\overline{C}_{\alpha,\mu}$ the constant appearing in \eqref{eq:overC-alp};
while for every $ \xi > \lambda$, by using \eqref{def.omg} and the change of variables, we infer that
\begin{align*}
  I_{i,3} \leq \widetilde{I}_{i,3} = c_\alpha \,\mathrm{p.v.}\int_{|z|\geq 2\xi}
  \frac{\delta_i \big(\log |\xi e_1 -z | -\log \xi \big)}{2|z|^{d+\alpha}} \dd z_h \dd z_1
  \leq c_\alpha \widetilde{C}_{\alpha,d}  \delta_i \xi^{-\alpha} ,
\end{align*}
where (noting that $|e_1 - z|\leq 2 |z|$)
\begin{align*}
  \widetilde{C}_{\alpha,d} := &\, \mathrm{p.v.} \int_{|z|\geq 2}
  \frac{ \log |e_1 - z| }{2|z|^{d+\alpha}} \dd z_h \dd z_1  \\
  \leq &\, \mathrm{p.v.} \int_{|z|\geq 2} \frac{\log(2|z|)}{2|z|^{d+\alpha}} \dd z
  = \sigma_{d-1} \,\mathrm{p.v.}\int_2^\infty \frac{\log(2r)}{2r^{1+\alpha}} \dd r < +\infty.
\end{align*}
Thus combining the above two estimates yields
\begin{equation}\label{es:Aalp-3}
  I_{i,3}\leq \widetilde{I}_{i,3} \leq
  \begin{cases}
    c_\alpha \sigma_{d-1} \overline{C}_{\alpha,\mu} \,\delta_i \lambda^{- \mu}
    \xi^{\mu - \alpha},\quad & \textrm{for}\;\; 0<\xi \leq \lambda, \\
    c_\alpha \widetilde{C}_{\alpha,d}  \,\delta_i \xi^{-\alpha}, \quad & \textrm{for}\;\; \xi \geq \lambda.
  \end{cases}
\end{equation}
Collecting \eqref{split:Aalp}, \eqref{es:Aalp-1}, \eqref{eq:overC-alp}, \eqref{es:Aalp-2}, \eqref{es:Aalp-3} yields the desired estimate
\eqref{es:Aalp}.
\end{proof}

\subsection{Evolution of the MOC on $u$}\label{subsec:GE-u}
Next, we provide general estimates that lead to \eqref{aim-equi} under the scenario \eqref{break.mode}--\eqref{break.mode1}. From equation $\eqref{eq:Euni}_2$, we observe that
\begin{align}\label{eq:par-t-u}
  &\partial_t\big(u(x)-u(y)\big)= -\big( {u\partial_{x_1}u }(x)-{u\partial_{x_1}u }(y)\big)
  + \big(\mathcal{C}_\alpha( u,\rho)(x)- \mathcal{C}_\alpha( u,\rho)(y)\big).
\end{align}
The first term on the right-hand side of \eqref{eq:par-t-u} represents the advection term, which can be estimated similarly to \eqref{es:N1} as
\begin{equation}\label{es.upxu}
  \big| {u\partial_{x_1}u }(x)-{u\partial_{x_1}u }(y)\big|
   \leq |u(x)-u(y)|\, \partial_\xi\omega_2 (\xi,t_1)
   = e^{- c_0\,t_1} |u(x)-u(y)|\, \omega_2'(\xi).
\end{equation}
We will apply different estimates to $|u(x)-u(y)|$ in different cases, as mentioned in Remark \ref{rmk:drift}.


Our main focus is on the latter term of \eqref{eq:par-t-u}. Without loss of generality, we can again assume that
\begin{align*}
  x=\big(\tfrac{\xi}{2},0,\dots,0\big), \quad y=\big(-\tfrac{\xi}{2},0,\dots,0\big).
\end{align*}
We split the term as follows:
\begin{align}\label{split:Calp}
  & \quad \mathcal{C}_\alpha( u,\rho)(x)- \mathcal{C}_\alpha( u,\rho)(y)\nonumber\\
  & =c_\alpha \,\mathrm{p.v.}\int_{\R^d} \frac{ \rho(x+z)\big(u(x+z)-u(x)\big) -
  \rho(y+z)\big(u(y+z)-u(y)\big) }{|z|^{d+\alpha}} \dd z \nonumber \\
  & = c_\alpha\,\int_{\{z: \rho(x+z) \leq \rho(y+z)\}}
  \frac{\rho(x+z)[(u(x+z)-u(y+z))-(u(x)-u(y))]}{|z|^{d+\alpha}} \dd z\nonumber\\
  &\quad  + c_\alpha \, \int_{\{z: \rho(x+z) \leq \rho(y+z)\}} \frac{\big(\rho(x+z)-\rho(y+z)\big) \big(u(y+z)-u(y)\big)}{|z|^{d+\alpha}}
  \mathrm{d}z \nonumber \\
  & \quad + c_\alpha\,\int_{\{z: \rho(x+z) > \rho(y+z)\}}
  \frac{\rho(y+z)[(u(x+z)-u(y+z))-(u(x)-u(y))]}{|z|^{d+\alpha}} \dd z\nonumber\\
  &\quad  + c_\alpha \, \int_{\{z: \rho(x+z) > \rho(y+z)\}}
  \frac{\big(\rho(x+z)-\rho(y+z)\big) \big(u(x+z)-u(x)\big)}{|z|^{d+\alpha}}
  \mathrm{d}z \nonumber \\
  & =: J_1 + J_2 + J_3 + J_4.
\end{align}

For terms $J_1$ and $J_3$, using the scenario \eqref{break.mode}--\eqref{break.mode1} and \eqref{eq:rho-MP}, we have
\begin{align*}
  J_1 + J_3 \leq - \underline{\rho}\, c_\alpha\,\mathrm{p.v.}\int_{\mathbb{R}^d} \frac{\omega_2(\xi,t_1) - u(x+z) + u(y+z)}{|z|^{d+\alpha}} \mathrm{d}z,
\end{align*}
and similar to the treatment of $D_{\alpha,1}(\xi)$ in \eqref{eq:D} above, we can infer that
\begin{align}\label{es:J2-13}
  J_1 + J_3 \leq - \underline{\rho} \,e^{-c_0 t_1}\, D_{\alpha,2}(\xi),
\end{align}
where $D_{\alpha,2}(\xi)$ is defined in \eqref{es:Dalp} and satisfies \eqref{es:Dalp2}.

For the terms $J_2$ and $J_4$, using the scenario \eqref{break.mode}--\eqref{break.mode1}, we have
\begin{align}\label{es:J2-24}
  J_2 + J_4 & = c_\alpha \,\int_{\{z: \rho(x+z) \leq \rho(y+z)\}}
  \big(\rho(y+z)-\rho(x+z) \big) \frac{u(x) - u(y+z) - \omega_2(\xi,t_1)}{|z|^{d+\alpha}} \dd z \nonumber \\
  & \quad + c_\alpha \, \int_{\{z: \rho(x+z) > \rho(y+z)\}}
  \big(\rho(x+z) - \rho(y+z) \big) \frac{u(x+z) - u(y) - \omega_2(\xi,t_1) }{|z|^{d+\alpha}} \dd z \nonumber \\
  & \leq c_\alpha \,
  \int_{\{z:\rho(x+z) \leq \rho(y+z)\}} |\rho(x+z) - \rho(y+z)|
  \frac{\omega_2(|\xi e_1 - z|,t_1) - \omega_2(\xi,t_1)}{|z|^{d+\alpha}} \dd z
  \nonumber \\
  & \quad + c_\alpha \,
  \int_{\{z:\rho(x+z) > \rho(y+z)\}} |\rho(x+z) - \rho(y+z)|
  \frac{\omega_2(|\xi e_1 + z|,t_1) - \omega_2(\xi,t_1)}{|z|^{d+\alpha}} \dd z
  \nonumber \\
  & \leq  c_\alpha \,e^{-c_0 t_1}
  \int_{\mathbb{R}^d} |\rho(x+ z) - \rho(y+z)|
  \frac{\omega_2(|\xi e_1 - z|) - \omega_2(\xi)}{|z|^{d+\alpha}} \dd z =: \mathcal{J}.
\end{align}

\begin{remark}\label{rmk:comparison}
 Let us compare the alignment force $\mathcal{C}_\alpha(u,\rho)$ and the linear fractional dissipation $\mathcal{C}_\alpha(u,1)$. We can repeat the calculation in \eqref{split:Calp} to estimate $\mathcal{C}_\alpha(u,1)(x)-\mathcal{C}_\alpha(u,1)(y)$ by setting $\rho\equiv1$. In this case, the terms $J_1+J_3$ still represent the dissipation, and we have the estimate \eqref{es:J2-13}. However, for $\mathcal{C}_\alpha(u,1)$, we find that $J_2+J_4=0$. Hence, the difference between the alignment force $\mathcal{C}_\alpha(u,\rho)$ and the linear fractional dissipation $\mathcal{C}_\alpha(u,1)$ is reflected in the term $\mathcal{J}$. It is crucial to control this term using the dissipation.
\end{remark}

The integrand in $\mathcal{J}$ exhibits a similar structure to $A_{\alpha,2}(\xi)$, which was defined in \eqref{def:A-alp}. One might expect that it can be controlled by $\omega_1(\xi)A_{\alpha,2}(\xi)$. However, this is not the case. To illustrate this point, we decompose $\mathcal{J}$ in a similar manner as \eqref{split:Aalp}, obtaining the following decomposition:
\begin{align}\label{split:K}
  \mathcal{J} & = c_\alpha \,e^{-c_0 t_1}
  \int_{|z|\geq 2\xi} |\rho(x+ z) - \rho(y+z)|
  \frac{\omega_2(|\xi e_1 - z|) - \omega_2(\xi)}{|z|^{d+\alpha}} \dd z \nonumber \\
  & \quad + c_\alpha \,e^{-c_0 t_1} \int_{|z|\leq 2\xi} |\rho(x+ z) - \rho(y+ z)|
  \frac{\omega_2(|\xi e_1 - z|) - \omega_2(|\xi-z_1|)}{|z|^{d+\alpha}} \dd z \nonumber \\
  & \quad + c_\alpha \,e^{-c_0 t_1}
  |\rho(x) - \rho(y)|\int_{|z|\leq 2\xi}
  \frac{\omega_2(|\xi - z_1|) - \omega_2(\xi)}{|z|^{d+\alpha}} \dd z \nonumber \\
  & \quad + c_\alpha \,e^{-c_0 t_1}
  \int_{|z|\leq 2\xi} \Big( |\rho(x+ z) - \rho(y+z)| - |\rho(x) -\rho(y)| \Big)
  \frac{\omega_2(|\xi - z_1|) - \omega_2(\xi)}{|z|^{d+\alpha}} \dd z \nonumber \\
  & = : J_5 + J_6 + J_7 + J_8.
\end{align}

For $J_5$, $J_6$, and $J_7$, we extract $\omega_1(\xi)$ and treat the remaining terms using similar estimates as for $A_{\alpha,2}(\xi)$. Specifically, for $J_5$, we observe that $\omega_2(|\xi e_1 -z|) - \omega_2(\xi) > 0$ for every $|z|\geq 2\xi$. Utilizing \eqref{break.mode1} and \eqref{def:tild-Ii3}, we deduce that
\begin{align}\label{es:K1}
  J_5 \leq c_\alpha \, e^{-c_0 t_1} \omega_1(\xi) \int_{|z|\geq 2\xi} \frac{\omega_2(|\xi e_1-z |) -\omega_2(\xi)}{|z|^{d+\alpha}} \dd z
  \leq e^{-c_0 t_1} \omega_1(\xi) \widetilde{I}_{2,3},
\end{align}
where $\widetilde{I}_{2,3}$ satisfies \eqref{es:Aalp-3}.
For $J_6$, we recall that $I_{2,2}$ is given by \eqref{split:Aalp}, and using \eqref{break.mode1}, we have
\begin{align}\label{es:K2}
  J_6 \leq e^{-c_0 t_1} \omega_1(\xi) I_{2,2},
\end{align}
where $I_{2,2}$ satisfies \eqref{es:Aalp-2}.
Regarding $J_7$, the estimate is similar to that of $I_{i,1}$, and by discarding the negative contributions in the estimate of $I_{i,1}$, we obtain
\begin{equation}\label{es:K3}
  J_7 \leq e^{-c_0 t_1} \omega_1(\xi)\times
  \begin{cases}
    c_\alpha C_{d,\alpha}\overline{C}_{\alpha,\mu} \delta_2 \lambda^{-\mu} \xi^{\mu-\alpha} ,
    \quad & \textrm{for}\;\; 0<\xi \leq \lambda, \\
    \frac{ c_\alpha C_{d,\alpha}}{\alpha} \delta_2 \xi^{-\alpha},\quad & \textrm{for}\;\;  \xi \geq \lambda .
  \end{cases}
\end{equation}
Gathering \eqref{es:K1}, \eqref{es:K2}, \eqref{es:K3} yields
\begin{align}\label{es:K1-2-3}
  J_5 + J_6 + J_7 \leq e^{-c_0 t_1} \omega_1(\xi) \times
  \begin{cases}
    C_2 \delta_2 \lambda^{-\mu} \xi^{\mu - \alpha},\quad & \mathrm{for}\;\; 0<\xi \leq \lambda, \\
    C_2 \delta_2 \xi^{-\alpha}, \quad & \mathrm{for}\;\; \xi \geq \lambda,
  \end{cases}
\end{align}
where $C_2>0$ is the constant appearing in Lemma \ref{lem:Aalp}.

The most challenging term is $J_8$, which is related to the dangerous singular integral \eqref{eq:difference} near $y=x$ (or $z=0$). We notice that in the estimate \eqref{eq:Ii1} for the corresponding term $I_{i,1}$ in $A_{\alpha,i}$, by symmetrizing $z_1$ around 0 and utilizing the concavity of $\omega_i$, we can obtain a favorable negative sign for the first term in \eqref{eq:Ii1}. However, this is no longer the case with the prefactor $|\rho(x+ z) - \rho(y+z)| - |\rho(x) -\rho(y)|$ in $J_8$, which does not have a specific sign. Therefore, we can only establish bounds on this term.
Using the triangle inequality, we have
\begin{equation}\label{es:K4}
\begin{split}
  J_8 & \leq c_\alpha \, e^{-c_0 t_1}
  \int_{|z|\leq 2\xi} \big|\rho(x+z)-\rho(x) - (\rho(y+z)- \rho(y)) \big|
  \frac{\big|\omega_2(|\xi -z_1|) - \omega_2(\xi)\big|}{|z|^{d+\alpha}} \dd z =: \mathcal{K}.
\end{split}
\end{equation}

The following Lemma provides estimates on the bound $\mathcal{K}.$
\begin{lemma}\label{lem:K5}
  Let $\alpha\in (0,2)$ and $\omega_i(\xi)$ ($i=1,2$) be the modulus of continuity given by \eqref{def:omeg1}
and \eqref{def.omeg2}. There exists a positive constant $C_3$ depends only on $\alpha$ and $d$ such that for every
$0<\xi \leq \lambda$,
\begin{equation}\label{es:K5-1}
  \mathcal{K} \leq C_3 e^{-c_0 t_1} \delta_1 \delta_2 \lambda^{-2} \xi^{2-\alpha},
\end{equation}
and for every $\lambda < \xi\leq \Xi_2$,
\begin{equation}\label{es:K5-2}
  \mathcal{K} \leq e^{-c_0 t_1}\times
  \begin{cases}
    C_3 \big( \delta_2 (\lambda^{-1}\xi)^{\alpha-1} + V_0 \big)\,\omega_1(\xi)\xi^{-\alpha},
    \quad & \mathrm{for}\quad 1<\alpha<2, \\
    C_3 (\delta_2 + V_0)\, \omega_1(\xi) \xi^{-\alpha}, \quad & \mathrm{for}\quad 0<\alpha \leq1.
  \end{cases}
\end{equation}
\end{lemma}

\begin{remark}\label{rmk:K}
 In the estimate \eqref{es:K5-2}, the term $\mathcal{K}$ is controlled by
  \[\mathcal{K}\lesssim\omega_1(\xi)\xi^{-\alpha},\]
 with a potentially large constant coefficient. On the other hand, the dissipation \eqref{es:J2-13} has a lower bound, as shown in \eqref{es:Dalp2},
 \[J_1+J_3\lesssim-\omega_2(\xi)\xi^{-\alpha}.\]
To ensure that $\mathcal{K}$ is controlled by the dissipation, we choose $\delta_1$ to be much smaller than $\delta_2$. This choice allows us to control $\mathcal{K}$ effectively through the dissipation term.

 This idea of choosing different values for $\delta_1$ and $\delta_2$ seems to play a critical role in controlling $\mathcal{K}$, particularly in the critical case when $\alpha=1$.
 \end{remark}

\begin{proof}[Proof of Lemma \ref{lem:K5}]
  For every $0<\xi \leq 2\lambda$, noting that
$\omega_i(\eta) \leq \omega_i'(0^+) \eta = \delta_i \lambda^{-1} \eta$ for every $\eta>0$,
we see that
\begin{align*}
  \mathcal{K} & \leq 2 c_\alpha e^{-c_0 t_1} \int_{|z|\leq 2\xi}
  \omega_1(|z|) \frac{\omega_2(|z|)}{|z|^{d+\alpha}} \dd z \\
  & \leq 2 c_\alpha e^{-c_0 t_1} \delta_1 \delta_2 \lambda^{-2}
  \int_{|z|\leq 2\xi} |z|^{-d-\alpha +2}\dd z
  \leq e^{-c_0 t_1} \tfrac{8 c_\alpha \sigma_{d-1}}{2-\alpha} \delta_1 \delta_2 \lambda^{-2} \xi^{2-\alpha} .
\end{align*}
In particular, by virtue of the fact that $\omega_1(\xi)\geq \omega_1(\lambda) = \frac{3}{4}\delta_1$
for every $\xi\geq \lambda$, the above estimate gives that for every $\lambda< \xi\leq 2\lambda$,
\begin{align*}
  \mathcal{K} \leq e^{-c_0 t_1} \tfrac{64 c_\alpha \sigma_{d-1}}{2-\alpha} \delta_2 \omega_1(\xi) \xi^{-\alpha}.
\end{align*}
For every $2\lambda \leq \xi \leq  \Xi_2$ (with no loss of generality assuming that $\Xi_2 > 2\lambda$),
we have the following splitting
\begin{equation*}
\begin{split}
  \mathcal{K} & \leq  2 c_\alpha  e^{-c_0 t_1}
  \int_{\{z: |z|\leq 2\xi, |z_1|\leq \frac{\xi}{2}\}} \omega_1(|z|)
  \frac{\big|\omega_2(|\xi -z_1|) - \omega_2(\xi) \big|}{|z|^{d+\alpha}} \dd z \\
  & \quad + 2 c_\alpha  e^{-c_0 t_1}
  \int_{\{z: |z|\leq 2\xi, |z_1| \geq \frac{\xi}{2}\}} \omega_1(|z|)
  \frac{\big|\omega_2(|\xi -z_1|) - \omega_2(\xi) \big|}{|z|^{d+\alpha}} \dd z
  = : \mathcal{K}_1 + \mathcal{K}_2.
\end{split}
\end{equation*}
Noting that $|\xi -z_1|\geq \frac{\xi}{2}\geq \lambda$ for every $|z_1|\leq \frac{\xi}{2}$ and $\xi\geq 2\lambda$,
and by using \eqref{def:omeg1}, \eqref{def.omeg2} and the change of variables,
we infer that for every $2\lambda \leq \xi \leq \Xi_2$,
\begin{align*}
  \mathcal{K}_1 & = 2 c_\alpha e^{-c_0 t_1} \int_{\{z: |z|\leq 2\xi, |z_1|\leq \frac{\xi}{2}\}}
  \omega_1(|z|) \frac{\delta_2 \big| \log |1 - \frac{z_1}{\xi} |  \big|}{2|z|^{d +\alpha}} \dd z \\
  & \leq c_\alpha e^{-c_0 t_1} \delta_2 \xi^{-1} \int_{|z|\leq 2\xi} \frac{\omega_1(|z|)}{|z|^{d-1+\alpha}} \dd z \\
  & \leq e^{-c_0 t_1} c_\alpha \sigma_{d-1} \delta_2 \xi^{-1} \int_0^{2\xi} \frac{\omega_1(\eta)}{\eta^\alpha} \dd \eta \\
  & \leq e^{-c_0 t_1} c_\alpha \sigma_{d-1} \delta_2 \xi^{-1}  \Big(\int_0^\lambda \frac{\delta_1 \lambda^{-1}\eta }{\eta^\alpha} \dd \eta
  + \omega_1(2\xi) \int_\lambda^{2\xi} \frac{1}{\eta^\alpha} \dd \eta \Big) \\
  & \leq e^{-c_0 t_1} c_\alpha \sigma_{d-1} \delta_2 \xi^{-1} \times
  \begin{cases}
    \frac{1}{2-\alpha} \delta_1 \lambda^{1-\alpha} +\frac{2}{\alpha-1} \omega_1(\xi) \lambda^{1-\alpha},
    \quad & \textrm{for}\quad 1<\alpha<2, \\
    \delta_1 +  2\omega_1(\xi) \log \frac{2\xi}{\lambda}, \quad & \textrm{for}\quad \alpha=1, \\
    \frac{1}{2-\alpha} \delta_1 \lambda^{1-\alpha} + \frac{2}{1-\alpha} \omega_1(\xi) (2\xi)^{1-\alpha},
    \quad & \textrm{for}\quad 0<\alpha<1  ,
  \end{cases}\\
  & \leq e^{-c_0 t_1} c_\alpha \sigma_{d-1}  \times
  \begin{cases}
    \frac{2}{(2-\alpha)(\alpha-1)} \delta_2 \omega_1(\xi) \xi^{-1} \lambda^{1-\alpha},
    \quad & \textrm{for}\quad 1<\alpha<2, \\
    4\big(\delta_2 + V_0 \big)\omega_1(\xi) \xi^{-1}, \quad & \textrm{for}\quad \alpha=1, \\
    \frac{6}{1-\alpha} \delta_2\omega_1(\xi) \xi^{-\alpha},
    \quad & \textrm{for}\quad 0<\alpha<1 ,
  \end{cases}
\end{align*}
where in the last inequality we have used the facts that $\delta_1 \leq \frac{4}{3}\omega_1(\xi)$,
and in particular for $\alpha=1$ (using \eqref{def:Xi2})
\[\log \tfrac{2\xi}{\lambda}\leq \log 2 + \log \tfrac{\Xi_2}{\lambda} \leq 1 + 2 \delta_2^{-1} V_0.\]

For $\mathcal{K}_2$, in view of the fact $|z|\geq |z_1|\geq \frac{\xi}{2}$ and the concavity of $\omega_i(\xi)$ ($i=1,2$), we have
\begin{align*}
  \mathcal{K}_2 & \leq 2 c_\alpha e^{-c_0 t_1} \int_{\{z: |z|\leq 2\xi, |z_1|\geq \frac{\xi}{2}\}}
  \omega_1(|z|) \frac{\omega_2(|z|)}{(\xi/2)^{d+\alpha}} \dd z \\
  & \leq e^{-c_0 t_1} 2^{2d+1+\alpha} c_\alpha \sigma_{d-1} \omega_1(2\xi) \omega_2(2\xi) \xi^{-\alpha} \\
  & \leq e^{-c_0 t_1} 2^{2d +3+\alpha} c_\alpha \sigma_{d-1} V_0\,
  \omega_1(\xi) \xi^{-\alpha},
\end{align*}
where in the last line we have used that $\omega_2(\xi) \leq \omega_2(\Xi_2) = V_0$.
Collecting the above estimates on $\mathcal{K}_1$ and $\mathcal{K}_2$ leads to the inequality \eqref{es:K5-2}
in the case $2\lambda \leq \xi \leq \Xi_2$, as desired.
\end{proof}

When $\alpha\in(1,2)$, the estimate in \eqref{es:K5-2} is not ideal, as $\lambda^{-1}\xi$ can be very big in the region $\xi\gg\lambda$. To overcome this difficulty, we derive a refined MOC for $\rho$, replacing $\omega_1$ in the estimate. This refinement utilizes the relation \eqref{eq:Grho}, which we recall here:
\begin{equation}\label{eq:Grho}
	\rho = \partial_{x_1}\Lambda^{-\alpha}u-\Lambda^{-\alpha}G.
\end{equation}
We state the following lemma, which demonstrates that the MOC of $\rho$ can be controlled by the MOCs of $u$ and $G$. The proof follows the approach in \cite[Lemma 4.4]{do2018global}, and it will be provided in the Appendix.

\begin{lemma}\label{lem.Omg}
Let $1<\alpha<2$. Suppose $u(t)$ obeys $\omega_2(\xi,t)$ defined in \eqref{def.omeg2t}.
Then, for any $\tilde{x}$, $\tilde{y}\in \R^d$ with $\tilde{\xi} = |\tilde{x} - \tilde{y}|$, we have
\begin{align}\label{es:Omg}
  |\rho(\tilde{x},t_1) - \rho(\tilde{y},t_1)| \leq  \, \widetilde{C}_4
  \bigg(\int_0^{\tilde{\xi}}\frac{\omega_2(\eta,t)}{\eta^{2-\alpha}}\dd\eta
  + \tilde{\xi}\int_{\tilde{\xi}}^\infty \frac{\omega_2(\eta,t)}{\eta^{3-\alpha}}\dd \eta\bigg)
  + C_0\overline{\rho} \|F_0\|_{L^\infty} \tilde{\xi},
\end{align}
where $\widetilde{C}_4>0$ depends only on $\alpha,d$ and $C_0>0$ an absolute constant. In particular, for every $\tilde\xi>\lambda$,
\begin{align}\label{es.Omg.11}
  |\rho(\tilde{x},t_1) - \rho(\tilde{y},t_1)| \leq
  \tfrac{2\widetilde{C}_4}{(\alpha-1)(2-\alpha)} \omega_2(\tilde\xi)\tilde\xi^{\alpha-1} + C_0\overline{\rho} \|F_0\|_{L^\infty} \tilde{\xi}.
\end{align}
\end{lemma}

By utilizing \eqref{es.Omg.11} to replace $\omega_1$ in the estimate \eqref{es:K5-2}, we obtain an improved estimate for $\mathcal{K}$.
\begin{lemma}\label{lem:K5-subc}
  Let $1<\alpha <2$ and $\omega_2(\xi,t)$ be the modulus of continuity given by \eqref{def.omeg2}.
There exists a constant $C_4>0$ depending only on $\alpha$ and $d$ such that for every $\lambda< \xi \leq \Xi_2$,
\begin{align}\label{es:K5-subc2}
  \mathcal{K} \leq e^{-c_0 t_1} C_4 \Big(\overline{\rho} \|F_0\|_{L^\infty}
  \omega_2(\xi) \lambda^{1-\alpha} + V_0\, \omega_2(\xi) \lambda^{-1} \Big).
\end{align}
\end{lemma}

%

\begin{proof}[Proof of Lemma \ref{lem:K5-subc}]
The proof of this lemma mainly relies on the following result, which states that other than the MOC $\omega_1(\cdot)$,
one can obtain an additional control on the quantity $|\rho(\tilde{x},t) - \rho(\tilde{y},t)|$ for every $\tilde{x},\tilde{y}\in \mathbb{R}^d$.

Now we can use \eqref{es.Omg.11} to replace the estimate of $\big|(\rho(x+z)-\rho(x)) - (\rho(y+z)-\rho(y))\big|$
in $\mathcal{K}$, so that for every $\lambda< \xi \leq \Xi_2$,
\begin{align}
  \mathcal{K} & \leq 2 c_\alpha \mathrm{p.v.} \int_{|z|\leq 2\xi}
  \Big(C_0 \overline{\rho}\|F_0\|_{L^\infty} |z| + \tfrac{2\widetilde{C}_4}{(\alpha-1)(2-\alpha)} \omega_2(|z|) |z|^{\alpha-1} \Big)
  \frac{\omega_2(|z|, t_1)}{|z|^{d+\alpha}} \dd z \nonumber \\
  & \leq 2 e^{-c_0 t_1} c_\alpha  \sigma_{d-1} \bigg( C_0 \overline{\rho} \|F_0\|_{L^\infty} \int_0^{2\xi}
  \frac{\omega_2(\eta)}{\eta^\alpha} \dd \eta
  + \tfrac{2\widetilde{C}_4}{(\alpha-1)(2-\alpha)}
  \int_0^{2\xi} \omega_2(\eta) \frac{\omega_2(\eta)}{\eta^2} \dd \eta \bigg)
  \nonumber \\
  & \leq 2 e^{-c_0 t_1} c_\alpha\sigma_{d-1} C_0 \overline{\rho} \|F_0\|_{L^\infty} \bigg( \int_0^\lambda
  \delta_2 \lambda^{-1}\eta^{1-\alpha} \dd \eta
  + \omega_2(2\xi)\int_\lambda^{2\xi} \eta^{-\alpha} \dd \eta \bigg)
  \nonumber \\
  & \quad + 2 e^{-c_0 t_1} c_\alpha\sigma_{d-1} \tfrac{2\widetilde{C}_4}{(\alpha-1)(2-\alpha)}
  \bigg(\int_0^\lambda \delta_2^2  \lambda^{-2}  \dd \eta
  + \omega_2^2(2\xi) \int_\lambda^{2\xi} \frac{1}{\eta^2} \dd \eta \bigg)
  \nonumber \\
  & \leq 2 e^{-c_0 t_1} c_\alpha\sigma_{d-1} C_0 \bar{\rho}\|F_0\|_{L^\infty}  \Big(\tfrac{1}{2-\alpha} \delta_2 \lambda^{1-\alpha} + 2\omega_2(\xi)\tfrac{1}{\alpha-1} \lambda^{1-\alpha} \Big)
  \nonumber \\
  & \quad + e^{-c_0 t_1} \tfrac{4\widetilde{C}_4 c_\alpha \sigma_{d-1} }{(\alpha-1)(2-\alpha)}
  \Big(\delta_2^2 \lambda^{-1} + 4\omega_2^2(\xi) \lambda^{-1} \Big)
  \nonumber \\
  & \leq e^{-c_0 t_1} \tfrac{32 c_\alpha \sigma_{d-1}}{(\alpha-1)(2-\alpha)} \Big(C_0\overline{\rho} \|F_0\|_{L^\infty}
  \omega_2(\xi) \lambda^{1-\alpha} + \widetilde{C}_4 V_0\, \omega_2(\xi) \lambda^{-1} \Big), \nonumber
\end{align}
where in the last line we have used the facts that $\delta_2 \leq \frac{4}{3} \omega_2(\xi)$
and $\omega_2(\xi) \leq \omega_2(\Xi_2) = V_0$.
\end{proof}


Therefore, collecting \eqref{aim-equi} and \eqref{eq:par-t-u}, \eqref{es.upxu}, \eqref{es:J2-13}, \eqref{es:J2-24}, \eqref{split:K}, \eqref{es:K1-2-3},
\eqref{es:K4} leads to that for every $\alpha\in (0,2)$ and $0<\xi \leq \Xi_2$,
\begin{equation}\label{eq:GE-u}
\begin{split}
  & \quad \partial_t\big(u(x,t) -u(y,t) \big)|_{t=t_1} + c_0\, \omega_2(\xi,t_1) \\
  & \leq e^{- c_0\,t_1} \Big(-\underline{\rho} D_{\alpha,2}(\xi)
  + |u(x) - u(y)| \omega_2'(\xi) + c_0 \, \omega_2(\xi) + \mathcal{K} \Big) \\
  & \quad + e^{-c_0 t_1} \times
  \begin{cases}
    C_2 \delta_2 \omega_1(\xi) \lambda^{-\mu} \xi^{\mu - \alpha},\quad & \mathrm{for}\;\; 0<\xi \leq \lambda, \\
    C_2 \delta_2 \omega_1(\xi) \xi^{-\alpha}, \quad & \mathrm{for}\;\; \xi \geq \lambda,
  \end{cases}
\end{split}
\end{equation}
where $D_{\alpha,2}(\xi)$ satisfies \eqref{es:Dalp2} and $\mathcal{K}$ given by \eqref{es:K4} satisfies the estimates in Lemmas \ref{lem:K5} and \ref{lem:K5-subc}.

\section{The subcritical regime: global well-posedness and asymptotic behavior}\label{sec:thm-sub}
In this section, we finalize the application of the MOC method and provide a proof of Theorems \ref{thm.main} and \ref{thm.flocking}
for the subcritical regime when $\alpha\in(1,2)$. The same result has been previously established in \cite{lear2021unidirectional},
and here we present an alternative proof employing a different analytical approach.

Our objective is to demonstrate that both $\rho(t)$ and $u(t)$ satisfy the MOCs $\omega_1(\xi)$ and $\omega_2(\xi,t)$, respectively, as defined in \eqref{def:omeg1} and \eqref{def.omeg2t}. By leveraging Lemma \ref{lem.brk}, our task reduces to proving the validity of \eqref{aim.1} and \eqref{aim-equi}. For the sake of simplicity, we assume $\delta_1=\delta_2=:\delta$, resulting in $\omega_1(\xi)=\omega_2(\xi)=:\omega(\xi)$.

We will now proceed to establish the validity of \eqref{aim.1} for any $x\neq y\in\T^d$ with $\xi=|x-y|\in (0,\Xi_1]$
(recalling $\Xi_1$ is given by \eqref{def:Xi1}). The proof can be divided into two parts.


\medskip\noindent $\bullet$
For every $0<\xi\leq \lambda$, we apply \eqref{def.omg}, \eqref{es:Dalp2}, and \eqref{es:Aalp} to \eqref{eq:GE-rho}, resulting in the following estimate:
\begin{align}
  &\partial_t \rho(x,t_1) - \partial_t \rho(y,t_1)
  \leq \delta\lambda^{-1-\mu} \xi^{1+\mu- \alpha }
  \Big(- \tfrac {C_1 \mu \underline{\rho}}{4}   +  C_2 \delta
  + \overline{\rho} \|F_0\|_{L^\infty} \lambda^\mu \xi^{\alpha-\mu}+ \nonumber\\
  &\qquad\qquad + \overline{\rho}^2  \|\nabla F_0\|_{L^\infty}
  \delta^{-1} \lambda^{1+\mu} \xi^{\alpha-\mu}
  + \tfrac{\overline{\rho}^3}{c_0} \|H_0\|_{L^\infty} \lambda^\mu
  \xi^{\alpha-\mu} + \delta \lambda^{-1+\mu}\xi^{\alpha-\mu}\Big).\label{Lip.rho.1}
\end{align}
Note that for the last term of \eqref{eq:GE-rho}, we have used
\begin{equation}\label{eq:uMOC}
	|u(x,t_1)-u(y,t_1)|\leq\omega_2(\xi,t_1) \leq \omega_2(\xi). 
\end{equation}
The right-hand side of \eqref{Lip.rho.1} can be made to be strictly negative by choosing $\delta$ and $\lambda$ small enough. Set $\mu=\frac{\alpha}{2}$. First, we choose $\delta$ such that $C_2\delta<\frac{C_1\alpha\underline{\rho}}{16}$.  All the rest terms are scaling subcritical and can be made small by choosing $\lambda$ small enough. Indeed, we have
\[\lambda^{\frac\alpha2}\xi^{\frac\alpha2}+\delta^{-1}\lambda^{1+\frac\alpha2}\xi^{\frac\alpha2}+\delta\lambda^{-1+\frac\alpha2}\xi^{\frac\alpha2}\leq \lambda^{\alpha-1}(1+\delta^{-1}).\]
Taking a small $\lambda$, depending on $\delta, \alpha$, and $C_1,\overline{\rho},\underline{\rho},\|F_0\|_{L^\infty},\|\nabla F_0\|_{L^\infty},\|H_0\|_{L^\infty}$, the rest terms can be made smaller than $\frac{C_1\alpha\underline{\rho}}{16}$.

\medskip\noindent $\bullet$ For every $\lambda<\xi\leq \Xi_1 = \lambda e^{2\delta^{-1} \overline{\rho}-\frac{3}{2}} $,
we apply \eqref{def.omg}, \eqref{es:Dalp2}, \eqref{es:Aalp}, and \eqref{eq:uMOC} to \eqref{eq:GE-rho} and use the facts that $\frac{3}{4}\delta\leq \omega(\xi)$.
 It yields
\begin{align}
  &\partial_t \rho(x,t_1) - \partial_t \rho(y,t_1)
  \leq \tfrac{\omega(\xi)}{\xi^\alpha}\Big(- \tfrac{C_1 \underline{\rho}}{2}
  + C_2 \delta +\overline{\rho} \|F_0\|_{L^\infty}\xi^\alpha+\nonumber\\
  & \qquad\qquad + \tfrac{4}{3} \overline{\rho}^2 \|\nabla F_0\|_{L^\infty}\delta^{-1}\xi^{1+\alpha}
  + \tfrac{4}{3}\tfrac{\overline{\rho}^3}{c_0} \|H_0\|_{L^\infty}\lambda^{-1}\xi^{1+\alpha} + \tfrac{\delta}{2}\xi^{\alpha-1}\Big).\label{Lip.rho.2} 
\end{align}
The right-hand side of \eqref{Lip.rho.2} can be made to be strictly negative by choosing $\delta$ and $\lambda$ small enough. First, we choose $\delta$ such that $C_2\delta<\frac{C_1\underline{\rho}}{4}$. All the rest terms are scaling subcritical and can be made small by choosing $\lambda$ small enough. Indeed, we have
\[\xi^\alpha+\delta^{-1}\xi^{1+\alpha}+\lambda^{-1}\xi^{1+\alpha}+\xi^{\alpha-1}\leq \lambda^{\alpha-1}\big(1+e^{4\delta^{-1}\overline{\rho}-3}(1+\delta^{-1})\big).\]
Taking a small $\lambda$, depending on $\delta, \alpha, \overline{\rho}$, and $C_1,\underline{\rho},\|F_0\|_{L^\infty},\|\nabla F_0\|_{L^\infty},\|H_0\|_{L^\infty}$, the rest terms can be made smaller than $\frac{C_1\underline{\rho}}{4}$.

Next we justify that \eqref{aim-equi} holds for any $x\neq y\in\T^d$ with $\xi=|x-y|\in (0,\Xi_2]$
(recalling $\Xi_2$ is given by \eqref{def:Xi2}). We also consider two cases.

\medskip\noindent $\bullet$
For every $0<\xi\leq\lambda$,
by using \eqref{def.omg} and \eqref{es:Dalp2}, \eqref{es:K5-1}, \eqref{eq:GE-u}, \eqref{eq:uMOC}, we have
\begin{align}
  & \quad \partial_t\big(u(x,t) -u(y,t) \big)|_{t=t_1} + c_0\, \omega_2(\xi,t_1) \nonumber\\
  & \leq e^{-c_0 t_1} \Big( - \tfrac{C_1 \mu \underline{\rho}}{4} \delta \lambda^{-1-\mu} \xi^{1+\mu- \alpha}
  + \delta^2 \lambda^{ -2} \xi  + c_0 \delta \lambda^{-1}\xi 
  + C_3 \delta^2 \lambda^{-2} \xi^{2-\alpha}
  + C_2 \delta^2 \lambda^{-1-\mu} \xi^{1+\mu- \alpha}   \Big) \nonumber\\
  & \leq e^{- c_0\,t_1} \delta \lambda^{-1-\frac{\alpha}{2}} \xi^{1-\frac{\alpha}{2}}
  \Big( - \tfrac{C_1 \alpha \underline{\rho}}{8} + \delta \lambda^{-1+\frac{\alpha}{2}} \xi^{\frac{\alpha}{2}}
  + c_0 \lambda^{\frac{\alpha}{2}} \xi^{\frac{\alpha}{2}} + C_3 \delta \lambda^{-1+\frac{\alpha}{2}} \xi^{1-\frac{\alpha}{2}}  + C_2 \delta \Big),\label{eq:Lip-u-1}
\end{align}
where in the last inequality, we take $\mu =\frac{\alpha}{2}$.
The right-hand side of \eqref{eq:Lip-u-1} can be made to be strictly negative by choosing $\delta$ and $\lambda$ small enough. First, we choose $\delta$ such that 
\[C_3 \delta \lambda^{-1+\frac{\alpha}{2}} \xi^{1-\frac{\alpha}{2}}  + C_2 \delta\leq (C_2+C_3)\delta<\tfrac{C_1\alpha\underline{\rho}}{16}.\]
The remaining two terms are scaling subcritical, we may pick $\lambda$ small such that
\[\delta \lambda^{-1+\frac{\alpha}{2}} \xi^{\frac{\alpha}{2}}
  + c_0 \lambda^{\frac{\alpha}{2}} \xi^{\frac{\alpha}{2}}\leq \delta\lambda^{\alpha-1}+c_0\lambda^\alpha<\tfrac{C_1\alpha\underline{\rho}}{16}.\]

\medskip\noindent $\bullet$
For every $\lambda < \xi \leq \Xi_2 = \lambda e^{2\delta_1^{-1}V_0 -\frac{3}{2}}$,
using \eqref{def.omg}, \eqref{es:Dalp2}, \eqref{es:K5-subc2}, \eqref{eq:GE-u}, \eqref{eq:uMOC},
we have
\begin{align}
  & \quad \partial_t\big(u(x,t) -u(y,t) \big)|_{t=t_1} + c_0\, \omega_2(\xi,t_1) \nonumber \\
  & \leq e^{-c_0 t_1} \tfrac{\omega(\xi)}{\xi^\alpha}\Big(- \tfrac{C_1 \underline{\rho}}{2}   + \tfrac{\delta}{2} \xi^{\alpha-1} + c_0\xi^\alpha 
   + C_4  \Big(\overline{\rho} \|F_0\|_{L^\infty}
  \lambda^{1-\alpha} + V_0\, \lambda^{-1} \Big)\xi^\alpha + C_2 \delta\Big). \label{eq:Lip-u-2b}
\end{align}
The right-hand side of \eqref{eq:Lip-u-2b} can be made to be strictly negative by choosing $\delta$ and $\lambda$ small enough. First, we choose $\delta$ such that $C_2\delta<\frac{C_1\underline{\rho}}{4}$. All the rest terms are scaling subcritical and can be made small by choosing $\lambda$ small enough. Indeed, we have
\[\delta\xi^{\alpha-1}+\xi^{\alpha}+(\lambda^{1-\alpha}+\lambda^{-1})\xi^{\alpha}\leq \lambda^{\alpha-1}\big(1+e^{4\delta^{-1}V_0-3}\big).\]
Taking a small $\lambda$, depending on $\delta, \alpha, V_0$, and $C_1,\underline{\rho},\overline{\rho},\|F_0\|_{L^\infty}$, the rest terms can be made smaller than $\frac{C_1\underline{\rho}}{4}$.

Collecting all the estimates above and applying Lemma \ref{lem.brk}, we obtain the desired Lipschitz bounds \eqref{eq:rho-Lip} and \eqref{eq:u-Lip}. In combination with the blowup criterion \eqref{blow0}, we conclude the global wellposedness of smooth solution for the system \eqref{eq:Euni} in the subcritical regime $1<\alpha <2$. 

Moreover, the estimate \eqref{eq:u-Lip}, together with Lemma \ref{lem:exp-decay} directly implies the exponential decay of the velocity \eqref{eq:flock1}. The strong flocking estimate \eqref{eq:flock2} follows, see e.g. \cite[Pg. 827]{lear2021unidirectional}.


\section{The critical regime: global well-posedness and asymptotic behavior}\label{sec:crit}
In this section, we delve into the critical regime characterized by $\alpha=1$. Given the critical scaling \eqref{eq:criticalscaling} for both $\rho$ and $u$, the task of establishing a global well-posedness theory becomes notably more challenging compared to the subcritical regime.

It is worth addressing a key challenge in applying the framework presented in Section \ref{sec:thm-sub}. With $\alpha=1$, certain terms in the estimates, such as those in equations \eqref{Lip.rho.1}, \eqref{Lip.rho.2}, \eqref{eq:Lip-u-1}, and \eqref{eq:Lip-u-2b}, transition from subcritical to critical. 
For instance, the last term in \eqref{Lip.rho.1} becomes $\delta$, which can not be made small by choosing a small $\lambda$. Nonetheless, we may still control the term by choosing $\delta$ small.

However, there is one critical term that does not become small through diminutive $\delta$ and $\lambda$ values. It is the penultimate term in \eqref{eq:Lip-u-2b}: $C_4V_0\lambda^{-1}\xi$, coming from the term $\mathcal{K}$. To compound the challenge, estimate \eqref{es:K5-subc2} is inapplicable in the case of $\alpha=1$ due to the coefficient $C_4$ growing infinitely large as $\alpha$ approaches $1$. In fact, it is well-known that the Reisz transform $\partial_{x_1}\Lambda^{-1}$ does not preserve the MOC $\omega_2(\xi)$, thereby precluding the validity of \eqref{es:K5-subc2} for the case of $\alpha=1$.

Our main idea is to simultaneously propagate the MOCs of $\rho$ and $u$. To control the term $\mathcal{K}$, we may use the fact that $\rho$ obeys the MOC $\omega_1$, that is preserved in time. This leads to a stronger bound \eqref{es:K5-2}. The aforementioned penultimate term in \eqref{eq:Lip-u-2b} becomes $C_3V_0$ (with a finite $C_3$, in oppose to an infinite $C_4$). However, it is still not guaranteed that this term can be controlled by the dissipation $\tfrac{C_1\underline{\rho}}{2}$, for arbitrary initial data.

In light of this, as elaborated in Remark \ref{rmk:K}, we introduce the relation:
\begin{align}\label{eq:del1-del2}
  \delta_1 = \kappa \delta_2,\quad \textrm{with some}\;\; \kappa\in(0,1) \;\;\textrm{chosen later}.
\end{align}
Thus $\omega_1(\xi) = \kappa\, \omega_2(\xi)$, where $\omega_1(\xi)$ and $\omega_2(\xi)$ are given by
\eqref{def:omeg1} and \eqref{def.omeg2}. By taking a small auxiliary parameter $\kappa$, we are able to control the aforesaid term by the dissipation, for any smooth initial data.

Let us repeat the estimates in Section \ref{sec:thm-sub}, using \eqref{eq:del1-del2}.

We first prove that \eqref{aim.1} holds for any $x\neq y\in\T^d$ with $\xi=|x-y|\in (0,\Xi_1]$. Similarly as \eqref{Lip.rho.1} and \eqref{Lip.rho.2}, we set $\mu=\frac12$ and get the following. 

\medskip\noindent $\bullet$ For every $0<\xi \leq \lambda$,
\begin{align}
  &\partial_t \rho(x,t_1) - \partial_t \rho(y,t_1)
  \leq \delta_1\lambda^{-\frac32} \xi^{\frac12}
  \Big(- \tfrac {C_1\underline{\rho}}{8}   +  C_2 \delta_1
  + \overline{\rho} \|F_0\|_{L^\infty} \lambda^{\frac12} \xi^{\frac12}+ \nonumber\\
  &\qquad\qquad + \overline{\rho}^2  \|\nabla F_0\|_{L^\infty}
  \delta_1^{-1} \lambda^{\frac32} \xi^{\frac12}
  + \tfrac{\overline{\rho}^3}{c_0} \|H_0\|_{L^\infty} \lambda^{\frac12}
  \xi^{\frac12} + \delta_2 \lambda^{-\frac12}\xi^{\frac12}\Big),\label{Lip.rho.3}
\end{align}

\medskip\noindent $\bullet$ For every $\lambda<\xi\leq \Xi_1 = \lambda e^{2\delta_1^{-1} \overline{\rho}-\frac{3}{2}} $,
\begin{align}
  &\partial_t \rho(x,t_1) - \partial_t \rho(y,t_1)
  \leq \tfrac{\omega_1(\xi)}{\xi}\Big(- \tfrac{C_1 \underline{\rho}}{2}
  + C_2 \delta_1 +\overline{\rho} \|F_0\|_{L^\infty}\xi+\nonumber\\
  & \qquad\qquad + \tfrac{4}{3} \overline{\rho}^2 \|\nabla F_0\|_{L^\infty}\delta_1^{-1}\xi^{2}
  + \tfrac{4}{3}\tfrac{\overline{\rho}^3}{c_0} \|H_0\|_{L^\infty}\lambda^{-1}\xi^{2} + \tfrac{\delta_2}{2}\Big).\label{Lip.rho.4} 
\end{align}
Note that the last terms in \eqref{Lip.rho.3} and \eqref{Lip.rho.4} are scaling critical, but can be made small as long as $\delta_2$ is small.
Following the same procedure as in Section \ref{sec:thm-sub}, we may take $\delta_2$, $\kappa$ and then $\lambda$ small enough to make sure the right hand-side of \eqref{Lip.rho.3} and \eqref{Lip.rho.4} are negative, finishing the proof of \eqref{aim.1}.

Next we prove that \eqref{aim-equi} holds for any $x\neq y\in\T^d$ with $\xi=|x-y|\in (0,\Xi_2]$.

\medskip\noindent $\bullet$ For every $0<\xi\leq\lambda$, arguing as \eqref{eq:Lip-u-1}, we have
\begin{align}
  & \quad \partial_t\big(u(x,t) -u(y,t) \big)|_{t=t_1} + c_0\, \omega_2(\xi,t_1) \nonumber\\
  & \leq e^{- c_0\,t_1} \delta_2 \lambda^{-\frac32} \xi^{\frac12}
  \Big( - \tfrac{C_1 \underline{\rho}}{8} + \delta_2 \lambda^{-\frac12} \xi^{\frac12}
  + c_0 \lambda^{\frac12} \xi^{\frac12} + C_3 \delta_1 \lambda^{-\frac12} \xi^{\frac12}  + C_2 \delta_1 \Big).\label{eq:Lip-u-3}
\end{align}
There are three terms $\delta_2 \lambda^{-\frac12} \xi^{\frac12}, C_3 \delta_1 \lambda^{-\frac12} \xi^{\frac12}$ and $C_2 \delta_1 $ that are critical, all of which can be made small by choosing $\delta_2$ and $\kappa$ small enough. The remain subcritical term can be made small by choosing $\lambda$ small enough. 

\medskip\noindent $\bullet$ For every $\lambda<\xi\leq \Xi_2=\lambda e^{2 \delta_2^{-1} V_0 -\frac{3}{2}}$, we follow \eqref{eq:Lip-u-2b} but replace the estimate on $\mathcal{K}$ by \eqref{es:K5-2}. This leads to
\begin{align}
  & \quad \partial_t\big(u(x,t) -u(y,t) \big)|_{t=t_1} + c_0\, \omega_2(\xi,t_1) \nonumber \\
  & \leq e^{-c_0 t_1} \Big(- \tfrac{C_1 \underline{\rho}}{2}\,\tfrac{\omega_2(\xi)}{\xi}   + \tfrac{\delta_2}{2\xi}\,\omega_2(\xi) + c_0\omega_2(\xi) 
   + C_3  \big(\delta_2+V_0 \big)\tfrac{\omega_1(\xi)}{\xi} + C_2 \delta_2\tfrac{\omega_1(\xi)}{\xi}\Big)\nonumber \\
  & \leq e^{-c_0 t_1} \tfrac{\omega_2(\xi)}{\xi}\Big(- \tfrac{C_1 \underline{\rho}}{2}   + \tfrac{\delta_2}{2} + c_0\xi 
   + C_3  \big(\delta_2+V_0 \big)\kappa + C_2 \delta_2\kappa\Big). \label{eq:Lip-u-2}
\end{align}
Notably, the most dangerous term $C_3V_0\kappa$ can be made small by choosing a small enough parameter $\kappa$. The rest of the terms can be controlled by the dissipation by taking $\delta_2$ and $\lambda$ small enough, similarly as before.

Thus, by choosing $\delta_2, \kappa$ and $\lambda$, the right hand-side of \eqref{eq:Lip-u-3} and \eqref{eq:Lip-u-2} can be made negative, finishing the proof of \eqref{aim-equi}.

Now we apply Lemma \ref{lem.brk} to obtain Lipschitz bounds on $\rho$ and $u$. Global well-posdness and asymptotic strong flocking behavior follows from the same argument in Section \ref{sec:thm-sub}. This completes the proof of Theorems \ref{thm.main} and \ref{thm.flocking}.


\section{The supercritical regime: refined regularity criterion}\label{sec:reg-cr}
In this section, our focus is on proving Theorem \ref{thm:reg-cr}, which concerns the refined regularity criterion for the system \eqref{eq:Euni} in the supercritical regime $0<\alpha<1$.

The main challenge in establishing a global well-posedness theory lies in controlling the advection term. It has a supercritical scaling under our framework. We impose an additional regularity criterion \eqref{eq:reg-cr} that allows us to obtain enough control to the advection term. Notably, our criterion \eqref{eq:reg-cr} only requires a certain H\"older regularity on $u$, which represents a significant improvement over existing criteria such as \eqref{blow0}.

We utilize the criterion \eqref{eq:reg-cr} to obtain an improved bound
\begin{align}\label{es:N1-2}
  |u(x)-u(y)| \leq \|u\|_{C^\sigma} \xi^\sigma.
\end{align}
for any $x, y\in\T^d$ and $\xi=|x-y|$. The bound \eqref{es:N1-2} replaces \eqref{eq:uMOC} in controlling the advection term.

In the following, we will verify \eqref{aim.1} and \eqref{aim-equi}. We make use of \eqref{es:N1-2} to handle the advection term, and we claim that the rest of the terms can be treated through the same procedure as the critical regime in Section \ref{sec:crit}.

Recalling that $\omega_1(\xi)$ and $\omega_2(\xi,t)$ are given by \eqref{def:omeg1} and \eqref{def.omeg2t} respectively, we also assume that \eqref{eq:del1-del2} holds and thus $\omega_1(\xi) = \kappa\, \omega_2(\xi)$ with a small parameter $\kappa\in(0,1)$ to be chosen later.

We first prove \eqref{aim.1} holds for any $x\neq y\in\T^d$ with $\xi=|x-y|\in (0,\Xi_1]$. 

\medskip\noindent $\bullet$ For every $0<\xi\leq\lambda$, similar to \eqref{Lip.rho.1}, we have 
\begin{align}
  &\partial_t \rho(x,t_1) - \partial_t \rho(y,t_1)
  \leq \delta_1\lambda^{-1-\mu} \xi^{1+\mu- \alpha }
  \Big(- \tfrac {C_1 \mu \underline{\rho}}{4}   +  C_2 \delta_1
  + \overline{\rho} \|F_0\|_{L^\infty} \lambda^\mu \xi^{\alpha-\mu}+ \label{Lip.rho.1.sup}\\
  &\qquad\qquad + \overline{\rho}^2  \|\nabla F_0\|_{L^\infty}
  \delta_1^{-1} \lambda^{1+\mu} \xi^{\alpha-\mu}
  + \tfrac{\overline{\rho}^3}{c_0} \|H_0\|_{L^\infty} \lambda^\mu
  \xi^{\alpha-\mu} + \|u(t_1)\|_{C^\alpha}\lambda^{\mu}\xi^{\sigma-(1-\alpha)-\mu}\Big),\nonumber
\end{align}
where we have used \eqref{es:N1-2} for the last term. Under the regularity assumption \eqref{eq:reg-cr}, $\sigma>1-\alpha$.
We set $\mu=\frac{\sigma-(1-\alpha)}{2}>0$. Then the last term in \eqref{Lip.rho.1.sup}
\begin{equation}\label{eq:u-sup-1}
\|u(t_1)\|_{C^\alpha}\lambda^{\mu}\xi^{\sigma-(1-\alpha)-\mu}\leq
\|u(t_1)\|_{C^\alpha}\lambda^{\sigma-(1-\alpha)}	
\end{equation}
is subcritical and can be made small by taking $\lambda$ sufficiently small. All other terms in \eqref{Lip.rho.1.sup} can be treated the same as \eqref{Lip.rho.1}.

\medskip\noindent $\bullet$ For $\lambda<\xi\leq \Xi_1 = \lambda e^{2\delta_1^{-1} \overline{\rho}-\frac{3}{2}}$, we follow \eqref{Lip.rho.2} and use \eqref{es:N1-2} to get
\begin{align}
  &\partial_t \rho(x,t_1) - \partial_t \rho(y,t_1)
  \leq \tfrac{\omega_1(\xi)}{\xi^\alpha}\Big(- \tfrac{C_1 \underline{\rho}}{2}
  + C_2 \delta_1 +\overline{\rho} \|F_0\|_{L^\infty}\xi^\alpha+\nonumber\\
  & \qquad\qquad + \tfrac{4}{3} \overline{\rho}^2 \|\nabla F_0\|_{L^\infty}\delta_1^{-1}\xi^{1+\alpha}
  + \tfrac{4}{3}\tfrac{\overline{\rho}^3}{c_0} \|H_0\|_{L^\infty}\lambda^{-1}\xi^{1+\alpha}\Big)+\|u(t_1)\|_{C^\alpha}\tfrac{\delta_1}{2}\xi^{\sigma-1}.\label{Lip.rho.2.sup} 
\end{align}
To control the last term, we use the fact $\omega_i(\xi)\geq\omega_i(\lambda)=\frac34\delta_i$ and obtain
\begin{equation}\label{eq:u-sup-2}
\|u(t_1)\|_{C^\alpha}\tfrac{\delta_i}{2}\xi^{\sigma-1}\leq \tfrac{\omega_i(\xi)}{\xi^\alpha}\cdot\tfrac{4}{3\delta_i}\cdot\|u(t_1)\|_{C^\alpha}\tfrac{\delta_i}{2}\xi^{\sigma-(1-\alpha)}
\leq\tfrac{\omega_i(\xi)}{\xi^\alpha}\cdot\tfrac23\|u(t_1)\|_{C^\alpha}e^{2\delta_i^{-1} \overline{\rho}-\frac{3}{2}}\cdot\lambda^{\sigma-(1-\alpha)}.	
\end{equation}
Taking $i=1$, we deduce that the last term in \eqref{Lip.rho.2.sup} is scaling subcritical, and can be controlled by the dissipation term by choosing $\lambda$ small enough.
All other terms in \eqref{Lip.rho.2.sup} can be treated the same as \eqref{Lip.rho.2}.

Next we show that \eqref{aim-equi} holds for any $x\neq y\in\T^d$ with $\xi=|x-y|\in (0,\Xi_2]$.

\medskip\noindent $\bullet$
For every $0<\xi \leq \lambda$, arguing similarly as \eqref{eq:Lip-u-3} and using \eqref{es:N1-2}, we obtain
\begin{align}
  & \quad \partial_t\big(u(x,t) -u(y,t) \big)|_{t=t_1} + c_0\, \omega_2(\xi,t_1) \nonumber\\
  & \leq e^{- c_0\,t_1} \delta_2 \lambda^{-1-\mu} \xi^{1+\mu-\alpha}
  \Big( - \tfrac{C_1 \mu \underline{\rho}}{16} + \|u(t_1)\|_{C^\sigma} \lambda^\mu \xi^{\sigma-(1-\alpha) -\mu} +\nonumber\\
  &\qquad\qquad\qquad~~+ c_0 \lambda^{\mu} \xi^{\alpha-\mu} + C_3 \delta_1 \lambda^{-1+\frac{\alpha}{2}} \xi^{1-\frac{\alpha}{2}}  + C_2 \delta_1 \Big).\label{eq:Lip-u-1-sup}
\end{align}
Taking $\mu=\tfrac{\sigma-(1-\alpha)}{2}$ and applying \eqref{eq:u-sup-1}, we know the advection term can be controlled by the dissipation by choosing a small $\lambda$. The remaining terms in \eqref{eq:Lip-u-1-sup} can be treated same as \eqref{eq:Lip-u-1}.

\medskip\noindent $\bullet$
For every $\lambda <\xi \leq \Xi_2 = \lambda e^{2\delta_2^{-1} V_0 -\frac{3}{2}}$, analogous to \eqref{eq:Lip-u-2}, we apply \eqref{es:K5-2} to the term $\mathcal{K}$ and use \eqref{es:N1-2} for the advection term. It yields
\begin{align}
  & \quad \partial_t\big(u(x,t) -u(y,t) \big)|_{t=t_1} + c_0\, \omega_2(\xi,t_1) \nonumber \\
  & \leq e^{-c_0 t_1} \tfrac{\omega_2(\xi)}{\xi^\alpha}\Big(- \tfrac{C_1 \underline{\rho}}{2} + c_0\xi^\alpha 
   + C_3  \big(\delta_2+V_0 \big)\kappa + C_2 \delta_2\kappa\Big)+e^{-c_0 t_1}\|u(t_1)\|_{C^\sigma} \tfrac{\delta_2}{2} \xi^{\sigma-1}. \label{eq:Lip-u-2-sup}
\end{align}
Applying \eqref{eq:u-sup-2} with $i=2$, we see that the last term in \eqref{eq:Lip-u-2-sup} is subcritical, and can be controlled by the dissipation term by choosing $\lambda$ small enough. All other terms in \eqref{eq:Lip-u-2-sup} can be treated similar as \eqref{eq:Lip-u-2}, by taking sufficiently small $\delta_2, \kappa$ and $\lambda$.

Now we apply Lemma \ref{lem.brk} and deduce Lipschitz bounds on $\rho$ and $u$, which then lead to global well-posedness and asymptotic strong flocking behavior.

We would like to emphasize that the parameters $\delta_i$ are independent of the a priori bound in \eqref{eq:reg-cr}. But $\lambda$ depends on $\|u\|_{L^\infty(\R_+;C^\sigma(\T^d))}$. Moreover, from \eqref{eq:u-sup-1} and \eqref{eq:u-sup-2}, we see the relation $\|u(t_1)\|_{C^\sigma}\lambda^{\sigma-(1-\alpha)}\lesssim1$. Hence, we pick
\[\lambda\sim \|u\|_{L^\infty(\R_+;C^\sigma(\T^d))}^{-\frac{1}{\sigma-(1-\alpha)}}.\]
Together with \eqref{eq:rho-Lip} and \eqref{eq:u-Lip}, we conclude with the Lipschitz bounds \eqref{eq:rho-Lip0} and \eqref{eq:u-Lip0}. This completes the proof of Theorem \ref{thm:reg-cr}.

\begin{remark}\label{rmk:supcrit}
 If we extend our regularity criterion \eqref{eq:reg-cr} to encompass the scale-invariant class
\begin{equation}\label{eq:ccrit}  
\sup_{t\in[0,T^{*})}\|u(t)\|_{ C^{1-\alpha}(\T^d)} < \infty,
\end{equation}
namely $\sigma=1-\alpha$, we unearth a distinct challenge. In this context, the advection terms in \eqref{Lip.rho.1.sup}, \eqref{Lip.rho.2.sup}, \eqref{eq:Lip-u-1-sup} and \eqref{eq:Lip-u-2-sup} become critical. Consequently, they can not be made small by choosing $\lambda$ sufficiently small. Furthermore, our earlier choice of $\mu=\tfrac{\sigma-(1-\alpha)}{2}=0$ loses its relevance. Should we adopt any $\mu>0$, the last term in \eqref{Lip.rho.1.sup} that reads $\|u(t_1)\|_{C^\alpha}\lambda^\mu\xi^{-\mu}$ becomes unbounded when $\xi\to0$. Therefore, our existing framework falls short of encapsulating the global well-posedness scenario under the assumption delineated in \eqref{eq:ccrit}.

The implications of whether \eqref{eq:ccrit} could potentially propel global well-posedness, or whether global well-posedness might be attainable without any a priori regularity criterion, constitute intriguing questions that warrant dedicated exploration in the realm of future investigations.

A notable special case worth mentioning is that of "parallel shear flocks", which was examined in \cite[Section 4.1]{lear2022global}. In this context, we consider shear velocities represented by 
\[
\rho= \rho(x_2,\cdots,x_d,t),\quad u = u(x_2,\cdots,x_d,t)
\]
in the system \eqref{eq:Euni}. Consequently, $\partial_t \rho =0$, causing the advection term in equation $\eqref{eq:Euni}_2$ to vanish. Sine the regularity assumption \eqref{eq:reg-cr} is exclusively utilized to address the advection terms  in \eqref{eq:Euni}, our framework assures global regularity and asymptotic strong flocking for parallel shear flocks in the fall range of $0<\alpha<2$, including the supercritical regime, without any a priori regularity criterion. 
\end{remark}

\section{Appendix: proof of Lemmas \ref{lem:D} and \ref{lem.Omg}}\label{sec:lemmas}

We first present the proof of Lemmas \ref{lem:D}.
\begin{proof}[Proof of Lemma \ref{lem:D}]
Below we sketch the proof of \eqref{es:Dalp2}.
For every $0<\xi\leq\lambda$,
due to the concavity of $\omega_i(\xi)$ given by \eqref{def:omeg1} or \eqref{def.omeg2}, we have
$\omega_i(\xi+2\eta)+ \omega_i(\xi-2\eta) - 2\omega_i(\xi)\leq 2\omega_i''(\xi)\eta^2$,
and we use the first integral term in \eqref{es:Dalp} to get that
\begin{align*}
  D_{\alpha,i}(\xi) \geq - C_1 \int_0^{\frac{\xi}{2}}\frac{2\omega_i''(\xi)\eta^2}{\eta^{1+\alpha}} \dd \eta
  =  - C_1\frac{2^{\alpha-1}}{2-\alpha}\omega_i''(\xi)\xi^{2-\alpha}
  =  C_1\frac{\mu(\mu+1)2^{\alpha-1} }{4(2-\alpha)} \delta_i \lambda^{-1-\mu}\xi^{1-\alpha+\mu}.
\end{align*}

For every $\xi >\lambda$, we keep the second integral term in \eqref{es:Dalp}, and using the fact
$\frac{3}{4}\delta_i \leq \omega_i(\xi)$, we have
\begin{align*}
  \forall \eta \geq \tfrac{\xi}{2},\quad \omega_i(2\eta+\xi)-\omega_i(2\eta-\xi)
  \leq \omega_i(2\xi)=\omega_i(\xi)
  + \tfrac{\delta_i}{2}\log 2 \leq \tfrac{3}{2}\omega_i(\xi),
\end{align*}
and thus
\begin{align*}
  D_{\alpha,i}(\xi) \geq C_1 \int_{\frac{\xi}{2}}^\infty \frac{\omega_i(\xi)}{2 \eta^{1+\alpha}} \dd \eta
  = \frac{C_1 2^{\alpha-1}}{\alpha} \omega_i(\xi)\xi^{-\alpha},
\end{align*}
as expected.
\end{proof}

Next we turn to the proof of Lemma \ref{lem.Omg}.
\begin{proof}[Proof of Lemma \ref{lem.Omg}]
By the relation \eqref{eq:Grho}, we have
\begin{align}\label{eq:rho-diff}
  \rho(\tilde{x})-\rho(\tilde{y}) = \big(\partial_{x_1}\Lambda^{-\alpha} u(\tilde{x})
  - \partial_{x_1}\Lambda^{-\alpha} u(\tilde{y}) \big)
  - \big (\Lambda^{-\alpha}G(\tilde{x}) - \Lambda^{-\alpha} G(\tilde{y}) \big).
\end{align}
Noting that (see e.g. \cite[Proposition 3.1]{miao2012global})
\begin{align*}
   \partial_{x_1} \Lambda^{-\alpha}u(x) = c_{\alpha,d} \,\mathrm{p.v.}
  \int_{\R^d} \frac{z_1}{|z|^{d+ 2 -\alpha}} u(x -z) \dd z,
\end{align*}
with $c_{\alpha,d} = \frac{\Gamma(\frac{d+2-\alpha}{2})}{2^{\alpha-1} \pi^{d/2} \Gamma(\frac{\alpha}{2})}$,
and by exactly arguing as \cite[Lemma 3.2]{miao2012global} (for the case $\alpha=1$ see also \cite{kiselev2007global}),
we get
\begin{align}\label{es:Omeg1}
  |\partial_{x_1}\Lambda^{-\alpha} u(\tilde{x}) - \partial_{x_1}\Lambda^{-\alpha} u(\tilde{y})|
  \leq \widetilde{C}_4 \bigg( \int_0^{\tilde\xi}\frac{\omega_2(\eta,t_1)}{\eta^{2-\alpha}}\dd\eta
  + \tilde{\xi}\int_{\tilde\xi}^\infty \frac{\omega_2(\eta,t_1)}{\eta^{3-\alpha}}\dd \eta \bigg).
\end{align}
For the second term on the right-hand side of \eqref{eq:rho-diff}, in view of the $L^\infty$-estimate of $G$ in \eqref{es:G}, we obtain that for $\alpha\in (1,2)$,
\begin{align*}
  |\Lambda^{-\alpha}G(\tilde{x})- \Lambda^{-\alpha}G(\tilde{y}) | \leq
  \|\Lambda^{-\alpha}G\|_{\dot W^{1,\infty}(\T^d)} |\tilde{x} - \tilde{y}| \leq C_0\|G\|_{L^\infty(\T^d)} \tilde{\xi}
  \leq C_0 \overline{\rho} \|F_0\|_{L^\infty} \tilde{\xi}.
\end{align*}
Combining the above two estimates yields the desired inequality \eqref{es:Omg}.

Next, 
we explicitly calculate the integral in \eqref{es:Omeg1} for every $\tilde{\xi} > \lambda$.
Direct calculation gives that
\begin{align*}
  \int_0^{\tilde\xi} \frac{\omega_2(\eta,t_1)}{\eta^{2-\alpha}} \dd\eta &
  \leq \int_0^\lambda \frac{\delta_2 \lambda^{-1}}{\eta^{1-\alpha}} \dd \eta
  + \int_\lambda^{\tilde\xi} \frac{\omega_2(\eta)}{\eta^{2-\alpha}} \dd \eta
  \leq
  \frac{1}{\alpha} \delta_2\lambda^{\alpha-1} + \frac{1}{\alpha-1} \omega_2(\tilde\xi)\tilde\xi^{\alpha-1}
  \leq \frac{ 2}{\alpha-1} \omega_2(\tilde\xi)\tilde\xi^{\alpha-1},
\end{align*}
and
\begin{align*}
  \tilde\xi\int_{\tilde\xi}^\infty \frac{\omega_2(\eta,t_1)}{\eta^{3-\alpha}} \dd\eta
  \leq \tilde\xi \bigg( \frac{1}{2-\alpha} \omega_2(\tilde\xi) \tilde\xi^{\alpha-2}
  + \int_{\tilde\xi}^\infty \frac{\delta_2}{2\eta^{3-\alpha}} \dd \eta\bigg)
  \leq \frac{2}{2-\alpha} \omega_2(\tilde\xi) \tilde\xi^{\alpha-1},
\end{align*}
where in the above we have used the fact that $\delta_2 \leq \frac{4}{3}\omega_2(\tilde\xi)$.
Collecting the above inequalities gives the desired estimate \eqref{es.Omg.11}.
\end{proof}

\vskip3mm

\noindent\textbf{Acknowledgments.}
The authors would like to express their gratitude to Alexander Kiselev and Roman Shvydkoy for their insightful and inspiring remarks.

\noindent\textbf{Funding.}
The authors acknowledge the support of National Key Research and Development Program of China No. 2020YFA0712900 (YL and LX), National Natural Science Foundation of China Nos. 12071043 (YL), 12001041, 12171031 (QM), 11771043, 12271045 (LX), Educational Commission Science Programm of Jiangxi Province Grant No. GJJ2200521 (YL), and National Science Foundation Grants DMS-2108264, DMS-2238219 (CT).

\vskip5mm

\bibliographystyle{plain}

\end{document}